\documentclass[a4paper, 11pt]{article}

\usepackage{ulem}
\usepackage{comment}
\usepackage{tikz-cd}

\usepackage{tikz}
\usepackage{amsmath, amssymb}
\usepackage{mathrsfs}
\usepackage{thm-restate}

\usepackage{amsthm}
\usepackage{color}
\usepackage{xcolor}
\usepackage{graphics}
\usepackage{mathtools}

\usepackage{xurl}

\newtheorem{defn}{Definition}[section]  
\newtheorem{rem}[defn]{Remark}

\theoremstyle{plain}
\newtheorem{thm}[defn]{Theorem}

\newtheorem{prop}[defn]{Proposition} 
\newtheorem{lem}[defn]{Lemma} 
\newtheorem{cor}[defn]{Collorary}

\newtheorem*{thm*}{Theorem}



\newenvironment{pf*}[1]{\proof[#1]}{\endproof}

\newcommand{\myclaim}{問題}
\newtheorem{claim}[defn]{\myclaim}

\newcommand{\lm}{\lambda}

\newcommand{\Lm}{\Lambda}
\newcommand{\fp}{\mathbb{F}_p}
\newcommand{\fpp}{\mathbb{F}_{p^2}}

\newcommand{\bC}{\mathbb{C}}

\newcommand{\bQ}{\mathbb{Q}}
\newcommand{\bN}{\mathbb{N}}
\newcommand{\bZ}{\mathbb{Z}}

\newcommand{\End}{{\operatorname{End}}}

\newcommand{\height}{\operatorname{ht}}

\newcommand{\AHfactnumber}{3.1}
\newcommand{\AHwhenfp}{2.13}

\title{The Lang-Trotter conjecture on average for genus-$2$ curves with $S_3$ reduced automorphism group }
\author{Chihiro Ando\thanks{Graduate School of Environment and Information Sciences, Yokohama National University.
E-mail: \texttt{ando.chihiro.zs@gmail.com} } \  and Shushi Harashita\thanks{Graduate School of Environment and Information Sciences, Yokohama National University.
E-mail: \texttt{harasita@ynu.ac.jp}}}
\date{}

\begin{document}

\maketitle

\begin{abstract}
    \noindent 
    For an elliptic curve $E$ over $\mathbb{Q}$ without complex multiplication, Lang and Trotter conjectured that the number of primes $p <X$ at which $E$ has a supersingular reduction is asymptotically equal to $c\sqrt{X}/\log X$, where $c>0$ is a constant depending only on $E$. While it remains an open question, an average estimation related to the Lang--Trotter conjecture was established by Fouvry and Murty. This result is called the Lang--Trotter conjecture on average. We extend the Lang--Trotter conjecture to curves of genus $2$ and obtain a similar result to the Lang--Trotter conjecture on average for the family of curves $C_{\lambda}:y^2=x(x-1)(x-{\lambda})(x-(\lambda-1)/{\lambda})(x-1/ (1-\lambda))$. These curves are characterized as curves of genus $2$ with reduced automorphism group containing symmetric group $S_3$. 
  \end{abstract}  

\section{Introduction}
 For an elliptic curve $E$ over $\mathbb{Q}$ without complex multiplication, Lang and Trotter \cite{LT} conjectured that the number of such primes is asymptotically equal to 
 ${c\sqrt{X}}/{\log X}$,
 where $c>0$ is a constant depending only on $E$. It is still an open question, however,  Elkies \cite{Elkies} proved a weaker result, namely that there are infinitely many supersingular primes. Furthermore, an average estimation related to the Lang--Trotter conjecture was established by Fouvry and Murty \cite{FM01}.  They calculated
\[
\frac{1}{AB} \sum_{\lvert a \rvert \le A} \sum_{\lvert b \rvert \le B}  \#\{p<X\mid E_{a,b} \textrm{ is supersingular at } p\}
    \sim \frac{4 \pi}{3} \frac{\sqrt{X}}{\log X}
\]
for $A,B > X^{1+ \epsilon}$, where $E_{a,b}$ is an elliptic curve defined by an equation $y^2=x^{3}+ax+b$ and $a,b$ run through the integers. They also showed that the above average estimation holds for $A,B > X^{1/2+ \epsilon}$ and $AB > X^{3/2 +\epsilon}$. This result is called \textit{the Lang--Trotter conjecture on average}. One of the key \textcolor{black}{points} in their proof is that the number of supersingular elliptic curves $E_{a,b}$ over $\fp$ ($p \neq2,3$) with $0 \le a,b < p $ can be written as
\[
\frac{p}{2} \cdot \left(h(-4p)+h(-p) \right)+O(p),
\]
where $h(-p)$ and $h(-4p)$ are the class numbers of primitive quadratic forms with discriminants $-p$ and $-4p$, respectively. We note that $E_{a,b}$ has a supersingular reduction at $p$ if and only if the trace of the Frobenius morphism of $E_{a,b}/\fp$, denoted by $a_p(E_{a,b})$, is zero and that the result by Fouvry and Murty is the case of $a_p(E_{a,b})=0$. It was extended to the case $a_p(E_{a,b})=r$ for any $r\in \bZ$ by David and Pappalardi \cite{DP} and the improved result was given by Baier \cite{SB}. 

In this paper, we extend the Lang--Trotter conjecture on average to curves of genus 2. Let $K$ be an algebraically closed field of characteristic $p \ne 2$. Curves of genus $2$ are
classified into 7 types \textcolor{black}{with respect to their reduced automorphism groups (cf.\ Igusa \cite[\S\ 8]{Igusa}, also see \cite[1.2]{IKO02}). Here, the \textit{reduced automorphism group} ${\rm RA}(C)$ of hyperelliptic curve $C$ is the quotient of the automorphism group of $C$ by the central subgroup generated by the hyperelliptic involution.} 
In our previous work \cite{Ando}, we obtained the Lang--Trotter conjecture on average for the family of genus-$2$ curves
\[
y^2=x(x-1)(x+1)(x-{\lambda}) \left(x-\frac{1}{\lambda}\right),
\]
\textcolor{black}{whose reduced automorphism groups contain Klein-$4$ group}.
In the present paper, we focus on another type of them, the genus-$2$ curves defined by the following equation with parameter $\lm \in K$:
\[
C_{\lm}:y^2=x(x-1)(x-{\lambda}) \left(x-\frac{\lambda-1}{\lambda}\right)\left(x-\frac{1}{1-\lambda}\right).
\]
 These genus-$2$ curves are characterized as having the reduced automorphism group ${\rm RA}(C_\lambda)$ containing symmetric group $S_3$. For almost all $\lambda$, it is known that ${\rm RA}(C_\lambda)$ is isomorphic to $S_3$. 
Let $p$ be a prime such that $p \ge 5$. We consider whether $C_{\lm}$ has a superspecial reduction at $p$.
Here, a curve $C$ over a field $K$ is \textit{superspecial} if its Jacobian variety is isomorphic to a product of supersingular elliptic curves over $\overline{K}$, the algebraic closure. 
Our aim is to compute the average of 
\[
{\phi}_{\lm}(X):= \# \lbrace p<X \mid C_{\lm} \text{ has a superspecial reduction at } p  \rbrace .
\]
for $\lm\in \bQ$. To this end, we firstly determine the value 
\[
{\psi}_{p} := \# \lbrace \lm \in \fp \mid C_{\lm} \text{ is superspecial}  \rbrace 
\]
for each prime $p$. 
The number of isomorphism classes of superspecial genus-$2$ curves over $\overline{{\mathbb F}_p}$ was computed for each of the 7 types by
Ibukiyama, Katsura and Oort \cite{IKO02}. In addition, the number of isomorphism classes of superspecial genus-$2$ curves that have models over $\fp$ was determined by Ibukiyama and Katsura \cite{IKadd} in terms of the class numbers and the type numbers. 
Moreover, Katsura and Oort \cite{KatsuraOort} obtained the number of $(a,b) \in {\overline{\fp}}^2$ 
such that the genus-$2$ curve $C_{a,b}:y^2=(x^2-1)(x^2-a)(x^2-b)$ is superspecial.
However, neither the number of $(a,b) \in {\fp}^2$ such that the genus-2 curve $C_{a,b}$ is superspecial has been determined, nor the number of $\lambda\in\fp$ such that $C_\lambda$ is superspecial.  The first theorem determines $\psi_p$ in terms of the class numbers $h(-p)$ and $ h(-3p)$. 
\begin{restatable}{theor}{thmA} \label{thm:A}
     Let $p\ge 5$ be a prime. Then we have 
    \begin{enumerate}
        \item [] ${\psi}_{p}=\frac{3}{2}h(-3p)$ if $p \bmod 4=1$,
        \item [] ${\psi}_{p}=0$ if  $p \bmod 12=7$,
        \item [] ${\psi}_{p}=3h(-p)$ if $p \bmod 12=11$.
    \end{enumerate}
\end{restatable}
To prove Theorem \ref{thm:A}, we use the result of Ibukiyama, Katsura and Oort \cite[\S 1.2]{IKO02} that $C_{\lm}$ is superspecial if and only if two elliptic curves 
\[
E_{{\Lm}^{\varepsilon}(\lm)}:y^2=x(x-1)\left(x-(1-\lambda)(\lm + \varepsilon \sqrt{{\lambda}^2-\lambda+1})^2 \right)
\]
for $\varepsilon = \pm 1$
 are supersingular. 
  Thus, it suffices to count the number of $\lambda \in \fp$ such that two associated elliptic curves $E_{\Lm^{-}(\lm)}$ and $ E_{{\Lm}^{+}(\lm)}$ are supersingular. \textcolor{black}{In fact, if either of the elliptic curves is supersingular, then both are supersingular (cf.\ \cite[Proposition 1.10]{IKO02}).  
To count such $\lm\in\fp$, a new and more elaborate argument was needed, as outlined below, because the defining equation contains radicals and differs from those treated in previous works (e.g. \cite{AT02}, \cite{Ando} and \cite{FM01}) on elliptic curves.}
 We firstly show that the $j$-invariants of supersingular elliptic curves $E_{\Lm^{-}(\lm)}$ and $ E_{{\Lm}^{+}(\lm)}$ are the roots of $P_{3p}(X) \bmod p$ if $p \equiv1 \bmod 4$ and are the roots of $P_{p}(X) \bmod p$ if $p \equiv3 \bmod 4$ and that conversely any root of $P_{3p}(X) \bmod p$ for $p \equiv1 \bmod4$ and any root of $P_{p}(X)\bmod p$ for $p \equiv11 \bmod 12$ is equal to the $j$-invariant of a supersingular elliptic curve $E_{\Lm^{\varepsilon}(\lm)}$ with some $\lm \in \fp$ and $\varepsilon\in\{\pm1\}$ (see Section \ref{the-important}). Here $P_{3p}(X)$ and $P_p(X)$ are the Hilbert class polynomials of discriminant $-3p$ and $-p$ respectively (see \S \ref{complex}). An important fact is that the degrees of the polynomials $P_{3p}(X)$ and $P_p(X)$ are $h(-3p)$ and $h(-p)$ respectively.  We require an elaborate argument, since the endomorphism rings of the elliptic curves are far from obvious. 
 We also use the factorization of $P_{3p}(X)$ and $P_p(X)$ mod $p$ to 
 determine the value $\psi_p$. The factorization of $P_p(X) \bmod p$ for $p \equiv3\bmod4$ was determined by Elkies \cite{Elkies} or Kaneko \cite[Appendix]{kaneko}.
The factorization of $P_{3p}(X)\bmod p$ for $p\equiv 1 \bmod 4$ is determined in the companion paper \cite{AH}, which is reviewed in Theorem \ref{factorization_P3p} below. We shall use this result in Section \ref{pfofa}. 

Using the value $\psi_p$, we compute the average of $\phi_{\lm}(X)$ for $\lm \in \bZ$ such that $|\lm|$ is less than a positive number $N>X^{1+\epsilon}$.
 \begin{restatable}{theor}{thmB}
 \label{thm:B}
    Let $X,\epsilon$ be positive numbers and $N$ a positive number that satisfies $N > X^{1+\epsilon}$. Then, as $X \rightarrow \infty$, we have
    \[
   \frac{1}{N} \sum_{\lvert \lm \rvert \le N} {\phi}_{\lm}(X) \sim \frac{(6+4\sqrt{3})\pi}{9} \frac{\sqrt{X}}{\log X}.
    \]
\end{restatable}
We use the idea of Fouvry and Murty, and this theorem is followed by a similar argument to our previous work \cite[\S 5]{Ando}.

Furthermore, we also calculate the average of $\phi_\lm(X)$ for $\lm\in \bQ$ such that $\height{(\lm)}<N$. Here, for $\lm=b/a$ with coprime numbers $a,b$, the height of $\lm$, denoted by $\height{(\lm)}$, is max$\{|a|,|b|\}$.
\begin{restatable}{theor}{thmC} \label{thm:C}
      Let $X,\epsilon$ be positive numbers and $N$ a positive number that satisfies $N > X^{1+\epsilon}$. Then, as $X \rightarrow \infty$, we have
    \[
   \frac{1}{N^2} \sum_{\lm \in \bQ,\ \height(\lm) \le N } {\phi}_{\lm}(X) \sim \frac{4(3+2\sqrt{3})}{3\pi} \frac{\sqrt{X}}{\log X}.
    \]
 \end{restatable}

Our paper is organized into 9 sections. In Section \ref{preliminaries} we recall some properties of elliptic curves with complex multiplication in addition to the Hilbert class polynomials and the modular polynomials
We also review genus-$2$ curves. These are crucial in the proof of Theorem \ref{thm:A}.  Furthermore, we recall an analytic result that we use in the proof of Theorem B. In Section \ref{isogeniesbetweenElaElad}, we construct isogenies between the two elliptic curves $E_{\Lm^{-}(\lm)}$ and $E_{\Lm^{+}(\lm)}$ and in Section \ref{supersingularellwithlambda} we give details of the properties of supersingular elliptic curves $E_{\Lm^{\pm}(\lm)}$ with $\lm \in \fp$ and finally determine their endomorphism rings. In Section \ref{with Endomorphism containing}, we conversely observe elliptic curves over $\overline{\fp}$ with a certain endomorphism ring. In Section \ref{the-important} we discuss the relationship between the $j$-invariants of the supersingular elliptic curves $E_{\Lm^{\pm}(\lm)}$ and the Hilbert class polynomials. 
We finally prove Theorem \ref{thm:A} in Section \ref{pfofa}, Theorem \ref{thm:B} in Section \ref{pfofb} and Theorem \ref{thm:C} in Section \ref{pfofc}. 


\subsection*{Acknowledgements}
This paper is a part of master’s thesis of the first author. She thanks
the second author for his supervision and for his helpful advice.  The authors also thank Toshiyuki Katsura for his helpful comments.
 
\section{Preliminaries}
\label{preliminaries}
In this section, we recall some known results on
elliptic curves with complex multiplication as well as some properties of the Hilbert class polynomials and the modular polynomials.  
We also 
recall some properties of genus-$2$ curves. Finally, we look at an analytic result for the Legendre symbol.

{
\subsection{The Hilbert class polynomials and modular polynomials}
\label{complex}
In this subsection we collect the facts on the Hilbert class polynomials and on the
modular polynomials which will be used in this paper.  A detailed exposition, together
with the results of Gross--Zagier \cite{GZ} on singular moduli, is given in the companion
paper \cite[\S 2]{AH}.

Let $D$ be a positive integer with $D\equiv0,3 \bmod 4$ and consider the imaginary
quadratic order of discriminant $-D$:
\[
O_D=\bZ\left[\frac{D+\sqrt{-D}}{2}\right].
\]
The \textit{Hilbert class polynomial of discriminant $-D$} is the monic polynomial $P_D(X)$
whose roots are precisely the distinct $j$-invariants of elliptic curves over
$\overline{\bQ}$ with complex multiplication by $O_D$.  It is well-known that the degree
of $P_{D}(X)$ is equal to the class number of $O_D$ and that $P_{D}(X)$ has its
coefficients in the ring of integers; hence it makes sense to consider $P_D(X)$ mod a
prime $p$.  The \textit{modular polynomial} $\Phi_n(X,Y)$ is the polynomial in $X$ and
$Y$ characterized by the condition that $\Phi_n(j(E_1),j(E_2))=0$ if and only if there
exists an isogeny $E_1 \to E_2$ of elliptic curves whose kernel is a cyclic group of
order $n$.

The next lemma of Deuring tells us when an elliptic curve $E$ over a number field has a
supersingular reduction at an odd prime $p$ of good reduction.
\begin{lem}[Deuring:\cite{Deu}]
\label{nonquadraticp}
Let $p$ be an odd prime of good reduction for $E$. Then $E$ has supersingular reduction at
$p$ if and only if there exists some $D \equiv 0 \textrm{ or }3 \mod 4$ such that
$p$ divides the numerator of $P_{D}(j(E))$ and either $\left(\frac{-D}{p}\right)=-1$ or the
highest power of $p$ dividing $D$ is odd.
\end{lem}
Deuring also proved that complex multiplication in characteristic $p$ can be lifted to
characteristic zero.
\begin{thm}[The Deuring lifting theorem, {\cite[Chapter 13, Theorem 14]{SL}}]
\label{deuring}
    Let $A_0$ be an elliptic curve in characteristic $p$, with an endomorphism $\alpha_0$
    which is not trivial. Then there exists an elliptic curve $A$ defined over a number
    field, an endomorphism $\alpha$ of A, and non-degenerate reduction of $A$ at a place
    $\mathfrak{B}$ lying above $p$, such that $A_0$ is isomorphic to $\overline{A}$, the
    reduction of $A$ at $\mathfrak{B}$, and $\alpha_0$ corresponds to $\overline{\alpha}$
    under the isomorphism.
\end{thm}
Finally we recall the factorization of $P_p(X) \bmod p$ for $p\equiv 3 \bmod 4$, obtained
by Elkies \cite{Elkies} (see also Kaneko \cite[Appendix]{kaneko} for another proof), and
the factorization of $P_{3p}(X) \bmod p$ for $p\equiv 1 \bmod 4$, which is the main
theorem of the companion paper \cite{AH}.
\begin{prop}[{Elkies:\cite{Elkies}}]
\label{fac_P_p}
Let $p$ be a prime with $p \equiv3 \bmod 4$. Then the Hilbert class polynomial $P_p(X)$
factors into $(X-12^3)R(X)^2$ mod $p$ where $R(X)\in\fp[X]$ is a separable polynomial
such that $R(12^3)\ne0 \bmod p$.
\end{prop}
\begin{thm}[{Ando and Harashita:\cite[Theorem \AHfactnumber]{AH}}]
\label{factorization_P3p}
    Let $P_{3p}(X)$ be the Hilbert class polynomial of discriminant $-3p$ where $p>5$ is a prime
    such that $p \equiv1 \bmod 4$. Then $P_{3p}(X) \bmod p$ factors into as follows:
    \[
    P_{3p}(X) \equiv a_1(X)^4a_2(X)^2b_{3,p}(X)^2\cdots b_{n,p}(X)^2 \pmod p,
    \]
    where
    \begin{enumerate}
        \item [] $a_1(X)=X-8000$ and $a_2(X)=X-54000$ if $p \equiv 5\bmod 24 $,
        \item  [] $a_1(X)=X-8000$ and $a_2(X)=1$ if $p \equiv 13\bmod 24 $,
          \item  [] $a_1(X)=1$ and $a_2(X)=X-54000$ if $p \equiv 17\bmod 24 $
          \item [] $a_1(X)=1$ and $a_2(X)=1$ if $p \equiv 1\bmod 24 $,
    \end{enumerate}
   and the polynomials $b_{i,p}(X)$ are pairwise distinct and each $b_{i,p}(X)$ is an irreducible polynomial of degree $2$ defined over $\fp$.
\end{thm}
We also quote the following result of \cite{AH}, which tells us when the
$j$-invariants of $E_{\Lm^{-}(\lm)}$ and $E_{\Lm^{+}(\lm)}$ coincide.  
\begin{prop}[{Ando and Harashita:\cite[Proposition \AHwhenfp]{AH}}]
\label{whenfp}
    Let $p>5$ be a prime and $E_{\Lm^{-}(\lm)}$ and $E_{\Lm^{+}(\lm)}$ supersingular
    elliptic curves with $\lm \in \fp$.  Suppose that
    $j(E_{\Lm^{-}(\lm)}) =j(E_{\Lm^{+}(\lm)})$.
    \begin{enumerate}
        \item [\rm{(a)}] If $p \equiv 1 \bmod 4$, then
        $j(E_{\Lm^{\pm}(\lm)}) \equiv8000$ or
$54000 \bmod p $,
        \item [\rm{(b)}] If $p \equiv 11 \bmod 12 $, then
        $j(E_{\Lm^{\pm}(\lm)}) \equiv54000 \bmod p$.
    \end{enumerate}
\end{prop}
The proof of Theorem \ref{factorization_P3p} in \cite{AH} uses the results of Sections
\ref{isogeniesbetweenElaElad}--\ref{the-important} of the present paper, but nothing
from Sections \ref{pfofa}--\ref{pfofc}; moreover, Sections
\ref{isogeniesbetweenElaElad}--\ref{the-important} themselves do not use Theorem
\ref{factorization_P3p} or Proposition \ref{whenfp}; hence there is no circular
reasoning.
}
\subsection{Superspecial genus-$2$ curves.}
Let $K$ be a field of characteristic $\ne 2$ and $\overline K$ the algebraic closure of $K$. We review some facts on genus-2 curves. Igusa \cite{Igusa} classified curves of genus $2$ into 7 types according to \textcolor{black}{their automorphism groups}.  We focus on one of \textcolor{black}{them}: the genus-$2$ curves \textcolor{black}{of the form} 
\[
C_{\lm}:y^2=x(x-1)(x-\lm)\left(x-\frac{\lm-1}{\lm}\right) \left(x-\frac{1}{1-\lm}\right).
\]
with $\lm\in K$. We note that this curve is nonsingular if and only if $\lm$ satisfies $\lm \ne 0,1$ and $\lm^2-\lm+1\ne 0$.  We recall that Ibukiyama, Katsura and Oort \cite{IKO02} give a condition equivalent to the genus-$2$ curve $C_{\lm}$ being superspecial.
For each $\lambda \in K$, we choose a square root (in $\overline K$) of $\lambda^2-\lambda+1$, fix it and write it as $\sqrt{\lambda^2-\lambda+1}$.
For the curve $C_{\lm}$, we consider two elliptic curves:
\[
E_{\Lambda^{\varepsilon}(\lm)}: Y^2=X(X-1)(X-(1-\lambda)(\lambda+\varepsilon\sqrt{{\lm}^2-\lambda+1})^2)
\]
for $\varepsilon=\pm 1$,
%
%
where we denote 
$
\Lm^{\varepsilon}(\lm):= (1-\lambda)(\lambda+\varepsilon\sqrt{{\lm}^2-\lambda+1})^2$.  We note that $\Lm^{\varepsilon}(\lm)\ne 0,1$ if and only if $\lm$ satisfies $\lm\ne0,1$.  The next proposition enables us to reduce the problem for the superspeciality of $C_\lm$ to that for the supersingularity of $E_{\Lambda^{-}(\lm)}$ and $ E_{{\Lambda}^{+}(\lm)}$.

\begin{prop}[Ibukiyama, Katsura and Oort {\cite[Lemma 1.1 (\romannumeral2),\ Proposition 1.3 (\romannumeral2), (\romannumeral3)]{IKO02}}]
\label{arigataya}
The curve $C_{\lm}$ is superspecial if and only if the two associated elliptic curves $E_{\Lambda^{-}(\lm)}$ and $ E_{{\Lambda}^{+}(\lm)}$ are supersingular.
\end{prop}
\subsection{Analytic results}
We review an estimation of a sum of the Legendre symbols weighted by the von Mangoldt function $\Lambda(n)$.
Let 
\[
S(D,X):=\sideset{}{^*} \sum_{|d| \le D} \left| \sum_{3 \le n \le X} \Lambda(n) \left( \frac{d}{n}\right) \right|,
\]
where the star on the summation indicates that $d$ is not  square. Then for every $C>0$ and $3 \le D \le X^{\frac{49}{50}}$, Fouvry and Murty \cite[Lemma 6]{FM01} showed, based on Jutila \cite[Lemma 8]{Ju}, that 
\[
S(D,X) \ll XD(\log X)^{-C}.
\]
In addition, Fouvry and Murty \cite[Lemma 6]{FM01} showed \textcolor{black}{that} this estimation remains valid if the variable of summation $n$ satisfies $n\equiv3\bmod 4$ (resp.\ $n\equiv 1 \bmod 4$).
We further provide a similar estimation in the case $n \equiv11\bmod 12$.
\begin{lem}
\label{jutila}
For $3 \le 3D \le X^{\frac{49}{50}}$, we have
\[
\sideset{}{^{**}}\sum_{0\le d \le D} \left| \ \sideset{}{''} \sum_{3 \le n \le X} \Lambda(n) \left( \frac{d}{n}\right) \right| \ll  3XD(\log X)^{-C},
\]
where the double star on the summation indicates that we sum over integers $d$ such that $d$ and $d/3$ are not square and the double prime indicates that $n \equiv 11 \bmod 12$. 
\end{lem}
\begin{proof}
   We can detect odd primes $n \equiv 11 \bmod 12$ by the function $\frac{1}{2} \left(1+(\frac{3}{n})\right)$, and then apply the estimation in the case $n \equiv3 \bmod 4$. 
%
    \end{proof}

\section{Isogenies between $E_{\Lm^{-}(\lm)}$ and $E_{\Lm^{+}(\lm)}$}
\label{isogeniesbetweenElaElad}
Let $K$ be a field with char$K\ne 2,3$. 
\textcolor{black}{Let $\lambda\in K$
with $\lm^2-\lm+1\ne 0$, $\lm\ne 0$ and $\lm \ne 1$.}
In this section, we \textcolor{black}{construct explicit isogenies} between the elliptic curves $E_{\Lambda^{-}(\lm)}$ and  $E_{\Lambda^{+}(\lm)}$, where
\[
E_{\Lambda^{\varepsilon}(\lm)}:y^2=x(x-1)\left(x-(1-\lm)(\lm+\varepsilon\sqrt{{\lm}^2-\lm+1} )^2 \right)
\]
for $\varepsilon = \pm 1$.
\textcolor{black}{These isogenies will play an important role, when we determine the endomorphism rings of $E_{\Lambda^{+}(\lm)}$ and $E_{\Lambda^{-}(\lm)}$ in the next section.} 

\textcolor{black}{We} firstly observe $3$-torsion points of the elliptic curves.
\begin{lem}
\label{x-three}
  For $\varepsilon=\pm1$, let $E_{\Lambda^{\varepsilon}(\lm)}$ be an elliptic curve with $\lm \in K$. Then
$(\lm+1+2\varepsilon\sqrt{{\lm}^2-\lm+1})/3$
   is the $x$-coordinate of a $3$-torsion point of $E_{\Lambda^{\varepsilon}(\lm)}$.
\end{lem}

\begin{proof}
    The division polynomial ${\psi}_3(x)$ of the elliptic curve $E_{\Lambda^{\varepsilon}(\lm)}$ is
    \[
    3{x}^4-4(1+\Lambda^{\varepsilon}(\lm)){x}^3+6\Lambda^{\varepsilon}(\lm) {x}^2 -{\Lambda^{\varepsilon}(\lm)}^2.
    \]
    By a direct calculation, we find that $(\lm+1+2\varepsilon\sqrt{{\lm}^2-\lm+1})/3$ is \textcolor{black}{a} root of ${\psi}_3(x)$ (see \cite[1]{Github}; the computations in this paper are carried out with Magma \cite{Magma}).
\end{proof}
We denote by $P^{\varepsilon}$ one of the $3$-torsion points of $E_{\Lambda^{\varepsilon}(\lm)}$ whose $x$-coordinate is $(\lm+1+2\varepsilon\sqrt{{\lm}^2-\lm+1})/3$. We now consider the quotient elliptic curves $E_{\Lm^{\varepsilon}(\lm)}/\{O,P^{\varepsilon},2P^{\varepsilon}\}$. In order to derive equations of them, we make a change of variables. 
\begin{lem}
\label{change_form3}
   For $\varepsilon=\pm1$, the elliptic curve $E_{\Lm^{\varepsilon}(\lm)}$ is isomorphic to an elliptic curve defined by $
    {Y_1}^2={X_1}^3+A^{\varepsilon}(\lm)\left(X_1-B^{\varepsilon}(\lm) \right)^2
    $ where 
   \begin{eqnarray*}
       && A^{\varepsilon}(\lm):=\left({\lm}^2-\lm+1\right)\left(2{\lm}-1+2\varepsilon\sqrt{{\lm}^2-\lm+1} \right),\\
       && B^{\varepsilon}(\lm):=-\frac{2({\lm}^2-\lm+1)(2\lm-1) +\varepsilon(5{\lm}^2- 5\lm  + 2 )\sqrt{{\lm}^2-\lm+1}}{9({\lm}^2-\lm+1)}.
   \end{eqnarray*}
   In addition, it is also isomorphic to an elliptic curve defined by
    \[
    {Y_2}^2={X_2}^3+\left(X_2+\frac{2}{27}-\varepsilon\frac{(\lm+1)(\lm-2)(2\lm-1)\sqrt{{\lm}^2-\lm+1}}{27({\lm}^2-\lm+1)^2} \right)^2.
    \]
\end{lem}

\begin{proof}
    Substituting $x=X_1+(\lm+1+2\varepsilon\sqrt{{\lm}^2-\lm+1})/3$ for $E_{\Lm^{\varepsilon}(\lm)}$ yields the equation of the first elliptic curve (see \cite[2-1,2-2]{Github}).  In addition, we find that $A^{\varepsilon}(\lm) \ne 0$. 
    Substituting $X_1=A^{\varepsilon}(\lm)X_2$ \textcolor{black}{and} $Y_1:={A^{\varepsilon}(\lm)}^{3/2}Y_2$ yields the equation of the second elliptic curve (see \cite[3-1,3-2]{Github}).
\end{proof}

    
Now we show that there exists \textcolor{black}{a} degree-$3$ \textcolor{black}{isogeny} between $E_{\Lambda^{-}(\lm)}$ and $E_{{\Lambda^{+}(\lm)}}$. 
\begin{prop}
For $\varepsilon=\pm1$, let $T^{\varepsilon}$ be the order-3 subgroup $\{P^{\varepsilon},2P^{\varepsilon},O\}$ of $E_{\Lambda^{\varepsilon}(\lm)}$. Then the quotient curve $E_{\Lambda^{\varepsilon}(\lm)}/T^{\varepsilon}$ is isomorphic to $E_{{\Lambda^{-\varepsilon}(\lm)}}$
\end{prop}

\begin{proof}
    From 
    \cite[\S 3]{Top} and Lemma \ref{change_form3}, the quotient curve $E_{\Lambda^{\varepsilon}(\lm)}/T^{\varepsilon}$ is given by the equation
    \[
    Y^2=X^3-27A^{\varepsilon}(\lm)\left(X-4A^{\varepsilon}(\lm)-27B^{\varepsilon}(\lm)\right)^2.
    \]
    Computing the $j$-invariants of $E_{\Lambda^{\varepsilon}(\lm)}/T^{\varepsilon}$ and $E_{{\Lambda^{-\varepsilon}(\lm)}}$ implies that they are isomorphic. (see \cite[4-1,4-2]{Github}) 
\end{proof}


\begin{rem}
\label{isogenybetE}
From the above proposition, we obtain an isogeny of degree-$3$ from $E_{\Lambda^{\varepsilon}(\lm)}$ to $E_{{\Lambda^{-\varepsilon}(\lm)}}$. We denote this isogeny by $\psi^{\varepsilon}$.  They are obtained by the composition of the following maps:
\footnotesize{
\begin{eqnarray*}
   i_{(\varepsilon)}: E_{\Lambda^{\varepsilon}(\lm)} &\overset{\cong}{\rightarrow}& E^{\varepsilon}:{Y}^2={X}^3+A^{\varepsilon}(\lm)\left(X-B^{\varepsilon}(\lm) \right)^2 \ ;\\ &&  X=x-\frac{\lm+1+2\varepsilon\sqrt{{\lm}^2-\lm+1}}{3},\ \ Y=y\\
   \pi_{(\varepsilon)}: E^{\varepsilon}&\rightarrow& E^{\varepsilon}/i_{(\varepsilon)}({T}^{\varepsilon}): z^2=u^3-27A^{\varepsilon}(\lm)\left( u-4A^{\varepsilon}(\lm)-27B^{\varepsilon}(\lm)\right)^2 \ ; \\ && u=\frac{3(6Y^2+6A^{\varepsilon}(\lm)B^{\varepsilon}(\lm)^2-3X^3-2A^{\varepsilon}(\lm)X^2)}{X^2},\\ && z=\frac{27Y\left(8A^{\varepsilon}(\lm)B^{\varepsilon}(\lm)^2-X^3-4A^{\varepsilon}(\lm)B^{\varepsilon}(\lm)X\right)}{X^3}\\
    j_{(\varepsilon)}:E^{\varepsilon}/i_{(\varepsilon)}({T}^{\varepsilon})&\overset{\cong}{\rightarrow}& E^{-\varepsilon}:{w}^2={v}^3+A^{-\varepsilon}(\lm)\left(v-B^{-\varepsilon}(\lm) \right)^2; \\
    && v={r_{-\varepsilon}}^2u, w={r_{-\varepsilon}}^3z, \text{ where } r_{-\varepsilon}:=\frac{2\lm-2\varepsilon\sqrt{{\lm}^2-\lm+1}-1}{9} \\
    {i_{(-\varepsilon)}}^{-1}:E^{-\varepsilon}&\overset{\cong}{\rightarrow}& E_{{\Lambda^{-\varepsilon}(\lm)}}:t^2=s(s-1)\left(s-(1-\lm)\left(\lm-\varepsilon\sqrt{{\lm}^2-\lm+1}\right)^2 \right);\\
    && s=v+\frac{\lm+1-2\varepsilon \sqrt{{\lm}^2-\lm+1}}{3},t=w
\end{eqnarray*}
}

\end{rem}
According to 
\cite[\S 3]{Top}, the kernel of composition $\psi^{-} \circ {\psi}^{+} $ (resp. ${\psi}^{+} \circ \psi^{-} $ ) consists of all $3$-torsion points. More precisely, the next lemma shows that $\psi^{-} \circ {\psi}^{+} $ is equal to $[-3]$ by a direct calculation. 

\begin{lem}
\label{minus-3}
    For the isogenies $\psi^{-}$ and ${\psi}^{+}$ that are defined above, the composition $\psi^{-} \circ {\psi}^{+} $ (resp. ${\psi}^{+} \circ \psi^{-} $ ) exactly corresponds to a multiplication-by-minus-3 map $[-3]$.
\end{lem}
\begin{proof}
    We show that $-{\psi}^{+}$ is the same map with a dual isogeny of $\psi^{-}$ which we denote by $\hat{\psi^{-}}$. By Remark \ref{isogenybetE}, it suffices to show that the two maps $-j_{(+)}\circ\pi_{(+)}$ and $\hat{\pi_{(-)}}\circ{j_{(-)}}^{-1}$ are the same. By 
    \cite[\S 3]{Top}, we can write the map $\hat{\pi_{(-)}}\circ {j_{(-)}}^{-1}:E^{+} \rightarrow E^{-}$  explicitly (see \cite[18]{Github}).
    Here we note that notations such as $r_+$ are those used in Remark \ref{isogenybetE} and that we have  $-27A^{-}(\lm)=A^{+}(\lm)/{r_+}^2,4A^{-}(\lm)+27B^{-}(\lm)=B^{+}(\lm)/{r_+}^2$. Since $1/r_+=-27r_{-}$, a direct calculation yields $-j_{(+)}\circ\pi_{(+)}=\hat{\pi_{(-)}}\circ{j_{(-)}}^{-1}$, which implies $-\psi^{+}=\hat{\psi^{-}}$. Hence $\psi^{-} \circ {\psi}^+={\psi}^+ \circ \psi^{-}=[-3]$.
\end{proof}

\section{Endomorphism rings of supersingular $ E_{\Lambda^{\pm}(\lm)}$ 
for $\lm \in \fp$}
\label{supersingularellwithlambda}
We observe supersingular elliptic curves $ E_{\Lambda^{-}(\lm)}$ and $ E_{{\Lambda^{+}(\lm)}}$ with $\lm \in \fp$ and consider their endomorphism rings. We note that $\lm\ne0,1$ because $ E_{\Lambda^{-}(\lm)}$ and $ E_{{\Lambda^{+}(\lm)}}$ are nonsingular.  Let $p$ be a prime. We assume that $p \ge 5$.
Firstly, we consider the case of $p \equiv1 \bmod 4$.

\begin{lem}
\label{num_of_fpp_1}
Let $p$ be a prime such that $p \equiv 1 \bmod 4.$
Assume that the elliptic curves $ E_{\Lambda^{-}(\lm)}$ and $ E_{{\Lambda^{+}(\lm)}}$ with $\lm \in \fp$ are supersingular. Then
\begin{enumerate}
    \item [\rm{(a)}]The square root $\sqrt{{\lm}^2-\lm+1}$ does not belong to $\fp$. In particular we have $\lm^2-\lm+1\ne 0$.
    \item[\rm{(b)}] The number of $\mathbb{F}_{p^2}$-rational points of $E_{\Lm^{\varepsilon}(\lm)}$ is equal to $(p-1)^2$ for $\varepsilon=\pm1$.
    \item[\rm{(c)}] The ${p^2}$ th-power Frobenius map $F$ sending $(x,y) $ to $(x^{p^2},y^{p^2})$ is equal to the multiplication-by-$p$ map $[p]$.
\end{enumerate}
\end{lem}

\begin{proof}
(a) If $\sqrt{\lm^2-\lm+1} \in \fp$, the elliptic curve $ E_{\Lambda^{\varepsilon}(\lm)}$ is defined over $\fp$, and {} all $2$-torsion points are $\fp$-rational. Therefore $4$ divides the number of $\fp$-rational points of the elliptic curve. However, since the elliptic curve is supersingular, we know that $\#E_{\Lm^{\varepsilon}(\lm)}(\fp)=p+1$, which is not divisible by $4$ because $p \equiv1 \bmod 4$.

(b)Auer and Top \cite[Proposition 2.2]{AT02}.

(c)
From (b), the trace of the ${p^2}$ th-power Frobenius map is $1+p^2-\#E_{\Lambda^{\varepsilon}(\lm)}(\mathbb{F}_{p^2})=2p$. Thus $F$ satisfies ${F}^2-2pF+p^2=0$ in $\text{End}\left(E_{\Lambda^{\varepsilon}(\lm)}\right)$ and therefore $F$ is equal to $[p]$. 
\end{proof}

The next Proposition tells us the property of the endomorphism rings of supersingular elliptic curves $E_{\Lm^{-}(\lm)}$ and $E_{\Lm^{+}(\lm)}$.
\begin{prop}
\label{end1}
    Let $p$ be a prime such that $p \equiv 1 \bmod 4.$ Assume that the elliptic curve $ E_{\Lambda^{-}(\lm)}$ and $ E_{{\Lambda^{+}(\lm)}}$ with $\lm \in \fp$ are supersingular. Then the endomorphism ring of $ E_{\Lambda^{\varepsilon}(\lm)}$ contains $\bZ \left[ \frac{1+\sqrt{-3p}}{2}\right]$ for $\varepsilon=\pm1$.

\end{prop}

\begin{proof}
    We may assume that $\varepsilon=-1$. Since $ E_{\Lambda^{-}(\lm)}$ is supersingular, the $p$ th-power Frobenius map $F$ satisfies $F^2=p$ from Lemma \ref{num_of_fpp_1} (c).  Our aim is to construct an endomorphism corresponding to $\sqrt{-3p}$ using the isogeny $\psi^-$ defined in Remark $\ref{isogenybetE}$ and Frobenius map $F$.  Since $\sqrt{{\lm}^2-\lm+1}$ is in $\fpp \smallsetminus \fp$, the $p$ th-power Frobenius map defines an isogeny between $ E_{\Lambda^{-}(\lm)}$ and  $E_{{\Lambda^{+}(\lm)}}$. Then we see that $ F \circ \psi^-$ defines an endomorphism of $ E_{\Lambda^{-}(\lm)}$. In addition, since $\psi^-$ and ${\psi}^+$ can be written as rational functions with the coefficients expressed by the combination of $\lm$ and $\sqrt{{\lm}^2-\lm+1}$, and we can obtain ${\psi}^{+}$ by changing the sign of $\sqrt{{\lm}^2-\lm+1}$ appearing in the rational functions that define $\psi^-$, we see that $F \circ \psi^- = {\psi}^+ \circ F$. Furthermore, from Lemma \ref{minus-3}, we find that $F \circ \psi^- \circ F \circ \psi^-= {\psi}^+ \circ F\circ F \circ \psi^- =[p]\circ{\psi}^+  \circ \psi^- =[-3p].$ 
    Finally, a map ${1+\sqrt{-3p}}$ is divisible by $[2]$ if and only if all $2$-torsion points of $E_{\Lm^{-}(\lm)}$ belong to the kernel of the map $1+\sqrt{-3p}$, which occurs precisely when $\sqrt{-3p}$ acts trivially on $ E_{\Lambda^{-}(\lm)}[2]$. Substituting $(x,y)=(0,0)$ and $(x,y)=(1,0)$ respectively into the rational function $s(x,y)$ in Remark \ref{isogenybetE}, we obtain $s(0,0)=0/(\lm-1)^4=0$ and $s(1,0)={\lm}^4/{\lm}^4=1$, which implies that $F \circ \psi^-(0,0)=(0,0)$ and that $F \circ \psi^-(1,0)=(1,0)$. Since $E_{\Lambda^{-}(\lm)}[2]\simeq \bZ/2\bZ\times \bZ/2\bZ$, we find that $\sqrt{-3p}$ acts trivially on $ E_{\Lambda^{-}(\lm)}[2]$. Hence the endomorphism ring of $ E_{\Lambda^{-}(\lm)}$ contains $\bZ \left[\frac{1+\sqrt{-3p}}{2}\right]$.
\end{proof}

\begin{rem}
\label{wheretomap}
    From the argument in Proposition \ref{end1}, we find that $\psi^\varepsilon$ maps the $2$-torsion point $(\Lambda^{\varepsilon}(\lm),0)$ to $({\Lambda^{-\varepsilon}(\lm)},0)$.
\end{rem}

Next, we consider the case of $p \equiv3 \bmod 4$. Our aim is to show that $\lm^2-\lm+1$ is square in ${\fp}^{*}$ and that $E_{\Lm^{-}(\lm)}$ and $E_{\Lm^{+}(\lm)}$ are defined over $\fp$ and to observe their endomorphism rings. We firstly prepare the lemma describing the number of rational points of $E_{\Lm^{-}(\lm)}$ and $E_{\Lm^{+}(\lm)}$.
\begin{lem}
\label{num_of_fpp_3}
Let $p$ be a prime such that $p \equiv 3 \bmod 4.$
Assume that the elliptic curves $ E_{\Lambda^{-}(\lm)}$ and $ E_{{\Lambda^{+}(\lm)}}$ with $\lm \in \fp$ are supersingular, and that $\sqrt{{\lm}^2-\lm+1}$ does not belong to $\fp$. Then, the number of $\mathbb{F}_{p^2}$-rational points of $E_{\Lm^{\varepsilon}(\lm)}$ is equal to $(p+1)^2$ for $\varepsilon=\pm1$. In addition the \textcolor{black}{${p^2}$-th power} Frobenius map sending $(x,y) $ to $(x^{p^2},y^{p^2})$ is equal to the multiplication-by-$-p$ map $[-p]$.
\end{lem}

\begin{proof}
    Follows from the same argument with Lemma \ref{num_of_fpp_1}.
\end{proof}

Now we are ready to show that $\lm^2-\lm+1$ is square in ${\fp}^{*}$
\begin{prop}
\label{pmod43fpcase}
    Let $p$ be a prime such that $p \equiv 3 \bmod 4.$ Assume that the elliptic curves $ E_{\Lambda^{-}(\lm)}$ and $ E_{{\Lambda^{+}(\lm)}}$ with $\lm\in\fp$ are supersingular. Then $\sqrt{{\lm}^2-\lm+1}$ belongs to $\fp$.
\end{prop}

\begin{proof}
    Assume that $\sqrt{{\lm}^2-\lm+1}$ did not belong to $\fp$. By Lemma \ref{num_of_fpp_3}, we follow a similar approach to Proposition \ref{end1} to find that the endomorphism ring of $ E_{\Lambda^{\varepsilon}(\lm)}$ for $\varepsilon=\pm1$ has a map that corresponds to $\sqrt{3p}$. We denote this map by $\phi$. Then $\phi$ satisfies ${\phi}^2-\text{tr}(\phi){\phi}+\text{deg}(\phi)=0$ in the endomorphism ring, and we find that $[\text{tr}(\phi)]\circ \phi=6p$. Comparing the degree of the maps, we obtain $12p=\text{deg}[\text{tr}(\phi)]$, which must be a square number. Hence $p$ should be $3$, but we assume that $p$ is greater than $3$. This yields a contradiction. 
\end{proof}
We provide further details in the case of $p\equiv7 \bmod 12$.
\begin{cor}
\label{peq7mod12}
    Let $p$ be a prime such that $p \equiv 7 \bmod 12.$ Then there is no $\lm \in \fp$  such that the elliptic curves $ E_{\Lambda^{-}(\lm)}$ and $ E_{{\Lambda^{+}(\lm)}}$ are supersingular.   
\end{cor}

\begin{proof}
    Assume that the elliptic curve $ E_{\Lambda^{\varepsilon}(\lm)}$ for $\varepsilon=\pm1$ is supersingular with a $\lm \in \fp$. Then, by Proposition \ref{pmod43fpcase}, the square root $\sqrt{{\lm}^2-\lm+1}$ belongs to $\fp$. Thus we have nontrivial $\fpp$-rational $3$-torsion points of $E_{\Lambda^{\varepsilon}(\lm)}$ from Lemma \ref{x-three}. This implies that $\#E_{\Lambda^{\varepsilon}(\lm)}(\fpp)[3]$ is divided by $3$, so is $\#E_{\Lambda^{\varepsilon}(\lm)}(\fpp)$. By Silverman \cite[\S 5 p.155 5.15]{AEC}, the number of $\fpp$-rational points of $E_{\Lambda^{\varepsilon}(\lm)}$ is $(p+1)^2$.  However, since $p \equiv 7 \bmod 12$, it is not divisible by $3$. \textcolor{black}{This is} a contradiction. 
\end{proof}
Now we obtain some properties of the endomorphism rings of $E_{\Lm^{-}(\lm)}$ and $E_{\Lm^{+}(\lm)}$ in the case of $p \equiv 3\bmod 4$. 
\begin{cor}
  Let $p$ be a prime such that $p \equiv 3 \bmod 4$.  Assume that the elliptic curves $ E_{\Lambda^{-}(\lm)}$ and $ E_{{\Lambda^{+}(\lm)}}$ with $\lm \in \fp$ are supersingular. Then the endomorphism ring of $ E_{\Lambda^{\varepsilon}(\lm)}$ for $\varepsilon=\pm1$ contains $\bZ \left[ \frac{1+\sqrt{-p}}{2}\right]$.
\end{cor}

\begin{proof}
    Since $ E_{\Lambda^{-}(\lm)}$ and $ E_{{\Lambda^{+}(\lm)}}$ are supersingular elliptic curves defined over $\fp$, we see that the composition $F \circ F$ corresponds to the map $[-p]$. In addition, all the $2$-torsion points of that elliptic curves are $\fp$-rational, so we find that the endomorphism ring contains $\bZ \left[ \frac{1+\sqrt{-p}}{2}\right]$.
\end{proof}

\section{Elliptic curves over $\overline{\fp}$ with the endomorphism ring containing $\bZ \left[ \frac{1+\sqrt{-3p}}{2}\right]$}
\label{with Endomorphism containing}
Let $p\ge 5$ such that $p\equiv1\bmod4$. In this section we show that any elliptic curve over $\overline{\fp}$ with the endomorphism ring containing $\bZ \left[ \frac{1+\sqrt{-3p}}{2}\right]$ \textcolor{black}{is isomorphic to}  $E_{\Lm^{\varepsilon}(\lm)}$ with some $\lm \in \fp$ and $\varepsilon\in\{\pm1\}$.

Let $E$ be an elliptic curve defined over $\overline{\fp}$ whose endomorphism ring contains $\bZ \left[ \frac{1+\sqrt{-3p}}{2}\right]$.  We \textcolor{black}{study} a map that corresponds to $\sqrt{-3p}$ in the endomorphism ring.

\begin{lem}
\label{compof-3pmap}
    A map that corresponds to the map $\sqrt{-3p}$ is the composition of Frobenius map and degree-$3$ isogeny. 
\end{lem}

\begin{proof}
    We denote by $\phi$ a map that corresponds to the map $\sqrt{-3p}$. The degree of $\phi$ is $3p$. Since multiplication-by-$p$ map $[p]$ is inseparable, the map $\phi$ is also inseparable and therefore we have ${\text{deg}}_i{\phi}=p$ and ${\text{deg}}_s{\phi}=3$. Thus, from Silverman \cite[Corollary 2.2]{AEC}, the map $\phi$ factors as
    \[
    E \overset{F}{\longrightarrow} E^{(p)} \overset{\nu}{\longrightarrow} E, 
    \]
    where the map $F$ is the $p$ th-power Frobenius map and the map $\nu$ is a separable isogeny of degree $3$.  
\end{proof}
For the isogeny $\nu$ of degree $3$ in the above lemma, we have $\# \text{Ker}(\nu)=\text{deg}(\nu)=3$, so we have $\text{Ker}(\nu)=\{O,Q,2Q\}$ with some order-3 point $Q\ne O$ of $E^{(p)}$. 
We note that $E$ is isomorphic to $E^{(p)}/\text{Ker}(\nu)$.
%
Furthermore, \textcolor{black}{we have}
\begin{prop}
    \textcolor{black}{The} elliptic curve $E$ is supersingular. 
\end{prop}
\begin{proof}
    From the Deuring lifting theorem, we have $P_{3p}(j(E))\equiv0 \bmod p$. Then we obtain the claim by Lemma \ref{nonquadraticp}.
\end{proof}
Changing the coordinates appropriately, we assume that $E$ is defined by the equation
\[
y^2=x^3+ax^2+bx+c
\]
with some $a,b,c \in \overline{\fp}$.  Since the Frobenius map $F:E \rightarrow E^{(p)}$ is a bijective homomorphism, there exists a unique point $P\ne O$ of $E$ such that $F(P)=Q$. The point $P$ is obviously a $3$-torsion point of $E$. Changing the coordinates appropriately, the coordinates of the point $P$ can be written as $(0,y_0)$. Then the coordinates of the point $Q$ can be written as $(0,y_1)$ where we set $y_1 ={y_0}^p$. We note that $y_0 \neq 0$ since $P$ is an order-$3$ point and that ${y_0}^2=c$. We denote by $T$ and $T'$ the order-$3$ groups $\{O,P,2P\} $ and $\{O,Q,2Q\} $ respectively. The tangent line at $P$ on $E$ is 
\[
y=\frac{b}{2y_0}x+y_0.
\]

Since $P$ is the $3$-torsion point, this line intersects $E$ at $P$ with multiplicity $3$, so we find that
\[ 
\left( \frac{b}{2y_0}x+y_0\right)^2 = ax^2+bx+c.
\]

Then we get  $4a{y_0}^2=b^2$. If $b=0$, then $a$ is also $0$. The next lemma tells us that this case does not occur.
\begin{lem}
    Under the above \textcolor{black}{notation, we have} $b\ne0$.
\end{lem}
\begin{proof}
    If $b =0$, the elliptic curve $E$ is isomorphic to an elliptic curve $y^2=x^3+1$ and $T=T'=\{O,(0,1),(0,-1)\}$. Since this elliptic curve is supersingular precisely when $p \equiv 2 \mod 3$, the case $b=0$ does not occur when $p \equiv 1 \bmod 3.$
    Thus, we may assume that $p \equiv 2 \bmod 3$.  Let $E':=E^{(p)}/T'$ and $\nu_0:E^{(p)}\rightarrow E'$ the isogeny of degree $3$ (cf.\ \cite[\S 3]{Top}).

    A map corresponding to $\sqrt{-3p}$ is the composition of the $p$ th-power Frobenius map $F$, degree-$3$ isogeny $\nu_0$ and some isomorphism $\iota :E'\rightarrow E$. By the assumption, the map $\sqrt{-3p}$ acts trivially on the group of the $2$-torsion points. We see that $E[2]=\{O,(-1,0),(\omega,0),(\overline{\omega},0)\}$ where $\omega$ is the root of an equation $x^2-x+1=0$ and $\overline{\omega}$ is its conjugate. We note that under the assumption that $p \equiv2 \bmod 3$ the square root $\sqrt{-3}$ does not belong to $\fp$.  It is a calculation to check that the map $\nu_0\circ F$ maps $(-1,0)$ as follows.
\begin{align*}
 E \overset{F}{\longrightarrow} E^{(p)} \overset{\nu_0}{\longrightarrow} E';
 (-1,0) \mapsto (-1,0) \mapsto (3,0)  
\end{align*}
We note that an elliptic curve $E'$ is defined by the equation $y^2=x^3-27$. Since the map $\iota \circ \nu_0 \circ F$ acts trivially on the group of the $2$-torsion points, the isomorphism $\iota:E'\rightarrow E$ should map $(3,0)$ to $(-1,0)$. Thus, $\iota$ should be the isomorphism sending $(x,y) $ to $(u^2x,u^3y)$ where $u=\sqrt{-1/3}$. \textcolor{black}{On the other hand, the map $\iota \circ \nu_0\circ F$
\[
E \overset{F}{\longrightarrow} E \overset{\nu_0}{\longrightarrow} E'\overset{\iota}{\longrightarrow} E
\]
sends $(\omega,0)$ as follows:
\[
 (\omega,0)\mapsto (\overline{\omega},0) \mapsto (3/{\overline{\omega}}^2,0) \mapsto  (-1/{\overline{\omega}}^2,0)= (\overline{\omega},0).
\]
Hence, it does not preserve the $2$-torsion point $(\omega,0)$, which implies a contradiction.}
\end{proof}
{
Thus $b \ne 0$, and then $a \ne 0$.  We fix $\sqrt{c}$ so that $\sqrt{c}=y_0$ and
$\sqrt{a}$ so that $b=-2\sqrt{a}\cdot\sqrt{c}$; then
$ax^2+bx+c=a\left(x-\sqrt{c}/\sqrt{a}\right)^2$, and the change of coordinates
$(x,y)\mapsto\left(x/a,\,y/(a\sqrt{a})\right)$ brings $E$ into the form
\[
E:y^2=x^3+(x-y_0)^2, \qquad \text{and hence} \qquad E^{(p)}:y^2=x^3+(x-y_1)^2,
\]
where we write again $y_0$ for the image $\sqrt{c}/(a\sqrt{a})$ of $y_0$, so that
$P=(0,y_0)$ and $Q=(0,y_1)$ with $y_1={y_0}^p$.
}
Let $\nu_0$ be the isogeny of degree $3$ from $E^{(p)}$ to $E^{(p)}/T'$ defined by the rational functions described in  \cite[\S 3]{Top}. Then the quotient curve $E^{(p)}/T'$ is defined by the equation
\[
y^2=x^3-27\left( x-(4+27y_1)\right)^2.
\]
We know that there is an isomorphism, say $\iota$, between $E$ and $E^{(p)}/T'$.  Let $Q'$ be a point of $E^{(p)}$ such that $\nu_0(Q')$ is the new $3$-torsion point.  That is, the point $Q'$ satisfies ${\nu}_0(\{O,Q',2Q'\})=\{O,(0,3\sqrt{-3}(4+27y_1)),(0,-3\sqrt{-3}(4+27y_1))\}$.  By  \cite[\S 3]{Top}, we see that the point $Q'$ is another three torsion point and that $E^{(p)}[3]$ is generated by $Q$ and $Q'$.  We set $P':=F^{-1}(Q')$.  Then we also see that $E[3]$ is generated by $P$ and $P'$.

\[
\begin{tikzcd}
 E\arrow[r, "F "] & E^{(p)} \arrow[r, "\nu"] \arrow[d,"{\nu}_{0}"] & E \\
& E^{(p)} / \{O,Q,2Q\} \arrow[ur, "\iota", dashed]
\end{tikzcd}
\]
{
Since $\iota \circ {\nu}_0 \circ F \circ \iota \circ {\nu}_0 \circ F=[-3p]$ and
$[-3p](E[3])=\{O\}$, the map $\iota \circ \nu_0$ sends $\{O,Q',2Q'\}$ to $T$, that is,
\[
\iota\left(\left\{O,\left(0,\pm3\sqrt{-3}(4+27y_1)\right)\right\}\right)
=T=\{O,(0,\pm y_0)\}.
\]
In particular $\iota$ maps the points with $x=0$ to the points with $x=0$, so that it is
given by $x \mapsto u^2x$, $y \mapsto u^3y$ with some $u \in \overline{\fp}$.  Hence
\[
x^3+\frac{1}{u^2}\left(x-\frac{y_0}{u^2} \right)^2=x^3-27\left(x-(4+27y_1) \right)^2,
\]
and comparing the coefficients of $x^2$ and of $x$ gives $u^2=-1/27$ and
$y_0/u^2=4+27y_1$, whence
\[
y_0+y_1=-\frac{4}{27}.
\]
Raising this relation to the $p$-th power, we find that $y_0={y_0}^{p^2}$, which means
that $y_0$ and $y_1(={y_0}^{p})$ belong to $\fpp$.
}

Now we are ready to show that $E$ is isomorphic to an elliptic curve $E_{\Lm^{\varepsilon}(\lm)}$ with some $\lm \in \fp$ and $\varepsilon\in\{\pm1\}$.



\begin{prop}
\label{lamdadekakeruyo}
    Let $E$ be an elliptic curve defined over $\overline{\fp}$ whose endomorphism ring containing $\bZ \left[ \frac{1+\sqrt{-3p}}{2}\right]$. Then it is isomorphic to an elliptic curve $E_{\Lm^{\varepsilon}(\lm)}$ with some $\lm \in \fp$ and $\varepsilon\in \{\pm1\}$.
\end{prop}

\begin{proof}
    From the above discussion, we assume that $E$ is defined by the equation $y^2=x^3+(x-y_0)^2$. In addition, we know that $y_0 \in \fpp$ and $y_0+{y_0}^p=-4/27$, so $y_0$ can be written as $-2/27+i$, where $i$ satisfies $i \in \fpp \smallsetminus {\fp}^{*}$ and ${i}^2 \in \fp$. We also know that $E_{\Lm^{\varepsilon}(\lm)}$ is isomorphic to an elliptic curve defined by the second equation in Lemma \ref{change_form3}, so $E$ is isomorphic to $E_{\Lm^{\varepsilon}(\lm)}$ with $\lm \in \overline{\fp}$ and $\varepsilon\in\{\pm1\}$ satisfying $\lm^2-\lm+1\ne 0$ and
    \begin{equation}
    \label{fp2part}
         i=\varepsilon\frac{(\lm+1)(\lm-2)(2\lm-1)\sqrt{{\lm}^2-\lm+1}}{27({\lm}^2-\lm+1)^2}.
    \end{equation}
   By squaring both sides of the equation, we obtain
    \[
    {i}^2= \frac{1}{{27}^2} \left( 4-27\frac{{\lm}^2(\lm-1)^2}{({\lm}^2-\lm+1)^3}\right),
    \]
    and therefore we find that
    \[
   2^8 \frac{({\lm}^2-\lm+1)^3}{{\lm}^2(\lm-1)^2}=2^8\frac{4-27^2{i}^2}{27}.
    \]
    The left-hand side is exactly the $j$-invariant of an elliptic curve $E_{\lm}:y^2=x(x-1)(x-\lm)$. We also find that $j(E_{\lm})$ is $\fp$-rational, so it is isomorphic to an elliptic curve defined by $y^2=x^3+ax+b$ with some $a,b \in \fp$. Depending on whether the polynomial $x^3+ax+b$ has no $\fp$-roots, a single $\fp$-root, or three $\fp$-roots, the $\overline{\fp}$-element $\lm$ can be ${\mathbb{F}}_{p^3}$-rational, of the form $\overline{\mu}/\mu$ with some $\mu \in \fpp\smallsetminus \fp$ where $\overline{\mu}$ is the $\fp$-conjugate of $\mu$, or $\fp$-rational, respectively. Here we may replace $\lm$ by any element of $[\lm]$ (with a suitable $\varepsilon$), where $[\lm]$ is as in Lemma \ref{num_iso}, since $j(E_\lm)$ and hence the square of the right-hand side of (\ref{fp2part}) (see below) are invariant under this replacement; in particular, in the second case we may assume that $\lm=\overline{\mu}/\mu$. However, $\lm$ cannot be in $ {\mathbb{F}}_{p^3} \smallsetminus\fp$ by the next Lemma \ref{not-fp3rational}. If the third case, we get the claim of the proposition. Assume that $\lm= \overline{\mu}/\mu$ with some $\mu \in \fpp\smallsetminus \fp$. 
{
For brevity we put
\[
t:=\mu+\overline{\mu},\qquad u:=2\overline{\mu}-\mu,\qquad
N:={\mu}^2-\mu\overline{\mu}+{\overline{\mu}}^2,\qquad s:=\sqrt{N},
\]
so that $t$ and $N$ belong to $\fp$.
With this notation, substituting $\overline{\mu}/\mu$ into the equation (\ref{fp2part}), we get
\[
i=-\varepsilon\frac{t\,u\overline{u}\,s}{27N^2}.
\]
Since $t$, $N$ and $u\overline{u}$ lie in $\fp$ while
$i \in \fpp \smallsetminus {\fp}^{*}$, the square root $s$ must lie in
$\fpp \smallsetminus {\fp}^{*}$.  We consider where the map $\nu\circ F$ sends the
$2$-torsion points.  Here $\nu\circ F$ is the map introduced above.
}
   \[
\begin{tikzcd}
 E\arrow[r, "F "] & E^{(p)} \arrow[r, "\nu"] \arrow[d,"{\nu}_{0}"] & E \\
& E^{(p)} / \{O,Q,2Q\} \arrow[ur, "\iota", dashed]
\end{tikzcd}
\] 
{
We may assume that $\varepsilon=-1$.  Then $E$ and $E^{(p)}$ are defined by
$y^2=x^3+(x-y_0)^2$ and $y^2=x^3+(x-\overline{y_0})^2$ respectively, where
$y_0=-2/27+i$.  We note that the equation defining $E^{(p)}$ is obtained by substituting
$\lm=\overline{\mu}/\mu$ and $\varepsilon=1$ for the second equation in Lemma
\ref{change_form3}.  Thus, the points
\[
S:=\left(-\frac{\lm+1-2\sqrt{\lm^2-\lm+1}}{3A^{-}(\lm)},0\right), \quad
S':=\left(-\frac{\lm+1+2\sqrt{\lm^2-\lm+1}}{3A^{+}(\lm)},0\right)
\]
are the $2$-torsion points of $E$ and $E^{(p)}$ respectively.  Substituting $\lm=\overline{\mu}/\mu$ into the
$x$-coordinates of $S$ and $S'$, and noting that $x(F(S))$ is obtained from $x(S)$ by
changing the sign of $s$ and interchanging $\mu$ and $\overline{\mu}$, we obtain
\[
x(S)=\frac{(t-2s)(u+2s)}{9N},\qquad
x(S')=\frac{(t+2s)(u-2s)}{9N}
\]
with $x(F(S))=(t+2s)(\overline{u}-2s)/(9N)$. Since the map $\nu \circ F$ preserves all $2$-torsion points, we see that
$\nu (F(S))=S$.  However, a direct calculation (see \cite[6-1,6-2]{Github}) shows that
$\nu(S')=S$.  We know that $\#\text{Ker}(\nu)=3$, and therefore we see that $F(S)=S'$.
On the other hand, since $u-\overline{u}=3(\overline{\mu}-\mu)$ we have
\[
x(S')-x(F(S))=\frac{(t+2s)(u-\overline{u})}{9N}
=\frac{(t+2s)(\overline{\mu}-\mu)}{3N},
\]
which does not vanish: indeed $\overline{\mu}\ne\mu$, and $t+2s=0$ would give
$t^2=4N=t^2+3(\mu-\overline{\mu})^2$, again contradicting $\mu\ne\overline{\mu}$.  This
leads to a contradiction.
}
Hence we obtain the claim. 
\end{proof}

\begin{lem}
\label{not-fp3rational}
   Let $p$ be a prime such that $p \equiv1 \bmod 4$. For an element $\lm \in {\mathbb{F}}_{p^3}$, suppose that $E_{\Lm^{-}(\lm)}$ and $E_{{\Lm^{+}(\lm)}}$ are supersingular. Then $\lm$ belongs to $\fp$.
\end{lem}

\begin{proof}
Assume that $\lm \in {\mathbb{F}}_{p^3} \smallsetminus\fp$. Our aim is to derive a contradiction. Since $E_{\Lm^{-}(\lm)}$ and $E_{\Lm^{+}(\lm)}$ are supersingular, by Auer and Top \cite[Proposition 2.2]{AT02}, we find that $\Lm^{-}(\lm)$ and ${\Lm^{+}(\lm)}$ are in $\fpp$.   Then we have $\Lm^{-}(\lm) \cdot {\Lm^{+}(\lm)} =(\lm-1)^4 \in \fpp$ and $(\lm-1)^4 \in {\mathbb{F}}_{p^3}$. Thus, $(\lm-1)^4 $ should be in $\fp$. Let $R$ be the subgroup of $\fp^*$ generated by $\left(\fp^*\right)^4$ and $(1-\lm)^4$.  In addition, we denote by $\fp(\sqrt[4]{R})$ the field generated by $\fp$ and $\{a \mid a^4\in R\}$.  From the assumption, we have ${\mathbb{F}}_{p^3}=\fp(\lm-1)$.  Moreover, we have $\fp(\sqrt[4]{R})=\fp(\lm-1)={\mathbb{F}}_{p^3}$. However, since $\mu_4\subset\fp$ for $p \equiv1 \bmod 4$, the Kummer theory tells us that $[\fp(\lm-1):\fp]=\left|R/\left(\fp^*\right)^4\right|\ne 3$, which implies a contradiction.  
\end{proof}

\section{$j$-invariants of supersingular $E_{\Lm^{-}(\lm)}$ and $E_{\Lm^{+}(\lm)}$ }
\label{the-important}
In this section, we observe the $j$-invariants of supersingular elliptic curves $E_{\Lm^{-}(\lm)}$ and $E_{\Lm^{+}(\lm)}$ with $\lm \in \fp$. In particular, we consider the relation between the $j$-invariants and the Hilbert class polynomials.

We firstly observe the case $p \equiv 1 \bmod 4$.
\begin{thm}
\label{important}
    Let $p\ge 5$ be a prime with $p\equiv 1 \bmod 4$ and $P_{3p}(X)$ the Hilbert class polynomial of discriminant $-3p$. Then $\overline{x_0} \in \overline{\fp}$ is the root of $P_{3p}(X) \bmod p$ if and only if an elliptic curve with $j$-invariant $\overline{x_0}$ is supersingular and can be defined by the following equation with some $\lm\in\fp$ and $\varepsilon \in \{\pm1\}$:
    \[
    y^2=x(x-1)(x-(1-\lm)(\lm+\varepsilon\sqrt{\lm^2-\lm+1})^2)
    \] 
\end{thm}
\begin{proof}
Assume that the elliptic curves $E_{\Lm^{\varepsilon}(\lm)}$ for $\varepsilon=\pm1$ with $\lm \in \fp$ are supersingular. Then Proposition \ref{end1} says that $E_{\Lm^{\varepsilon}(\lm)}$ has an endomorphism $(1+\sqrt{-3p})/2$. By Deuring's lifting theorem and the maximality of the order $O_{3p}$, there exists a $\alpha\in \bC$ and a prime ideal $\mathfrak{B}_0$ such that $P_{3p}(\alpha)=0$ and $\alpha \equiv j(E_{\Lm^{\varepsilon}(\lm)}) \bmod \mathfrak{B}_0$. Hence $P_{3p}(j(E_{\Lm^{\varepsilon}(\lm)}) ) \equiv0 \bmod p$. Conversely, assume that $\overline{x_0} \in \overline{\fp}$ is the root of $P_{3p}(X) \bmod p$.  Let $K$ be a splitting field of $P_{3p}(X)$ over $\bQ$. Then there exists a root $x_0$ of $P_{3p}(X)$ and a prime ideal $\mathfrak{B}$ lying above $p$ such that $\overline{x_0} \equiv x_0\bmod \mathfrak{B}$. The root $x_0$ is a $j$-invariant of an elliptic curve with endomorphism $(1+\sqrt{-3p})/2$, therefore the elliptic curve with $j$-invariant $\overline{x_0}$ also has an endomorphism $(1+\sqrt{-3p})/2$. Hence Proposition \ref{lamdadekakeruyo} tells us that there is an element $\lm \in \fp$ and $\varepsilon\in\{\pm1\}$ such that $\overline{x_0}=j(E_{\Lm^{\varepsilon}(\lm)})$. 
\end{proof}

Secondly, we observe the case $p \equiv 11 \bmod 12$

\begin{thm}
\label{important2}
    Let $p\ge 5$ be a prime such that $p \equiv 11 \bmod 12$ and $P_{p}(X)$ the Hilbert class polynomial of discriminant $-p$. Then $\overline{x_0} \in \overline{\fp}$ is the root of $P_{p}(X) \bmod p$ if and only if an elliptic curve with $j$-invariant $\overline{x_0}$ is supersingular and can be defined by the following equation with some $\lm\in\fp$ and $\varepsilon \in \{\pm1\}$:
    \[
    y^2=x(x-1)(x-(1-\lm)(\lm+\varepsilon\sqrt{\lm^2-\lm+1})^2)
    \] 
\end{thm}
\begin{proof}
    By Proposition \ref{pmod43fpcase}, for $\varepsilon=\pm1$ if $E_{\Lm^{\varepsilon}(\lm)}$ with $\lm \in \fp$ is supersingular, then $\Lm^{\varepsilon}(\lm)\in \fp$. Thus, the elliptic curve is defined over $\fp$ and the $p$ th-power Frobenius map $F$ acts trivially on all $2$-torsion points of $E_{\Lm^{\varepsilon}(\lm)}$. Hence its endomorphism ring contains $\bZ\left [\frac{1+\sqrt{-p}}{2}  \right]$ and thus the $j$-invariant of this elliptic curve is the root of $P_p(X) \bmod p$.

    Let $\overline{x_0}$ be the root of $P_p(X) \bmod p$ and $E$ an elliptic curve with $j$-invariant $\overline{x_0}$. The endomorphism ring of $E$ contains $\bZ\left [\frac{1+\sqrt{-p}}{2}  \right]$ so it is supersingular and defined over $\fp$; indeed, the endomorphism corresponding to $\sqrt{-p}$ is purely inseparable of degree $p$, so it is the $p$-th power Frobenius map of a model of $E$ over $\fp$ up to an isomorphism, and this Frobenius map acts trivially on $E[2]$ since $(1+\sqrt{-p})/2\in\End(E)$, whence $E[2]\subset E(\fp)$. Moreover, we may assume that it is given by the equation
    \[
    y^2=x(x-1)(x-t)
    \]
    with some $t \in \fp$. Our aim is to find a $\lm$ and $\varepsilon\in\{\pm1\}$ such that $\Lm^{\varepsilon}(\lm)=t$. From the assumption that $p \equiv 11\bmod 12$, we write $p=12m+11$ with some integer $m$. Since $E$ is supersingular elliptic curve defined over $\fp$, we have $\#E(\fp)=p+1=12(m+1).$ Thus, as an abstract group, $E(\fp)$ has a subgroup of order $3$. That is, $E$ has at least one $\fp$-rational point of order $3$, say $P=(a,b)$ with $a,b \in \fp$. We note that $a$ and $b$ satisfies two equations
\begin{eqnarray}
     b^2 &=& a(a-1)(a-t), \label{re1}\\
     3a^4-4(1+t)a^3+6ta^2-t^2 &=& 0.\label{re2}
\end{eqnarray}
 Let 
\begin{eqnarray*}
    f(a,b^2)&:=&- 3a^9 + 14a^8 - 34a^7 +
    50a^6 - 42a^5 + 21a^4b^2 + 18a^4 \\
    &&- 32a^3b^2 - 3a^3  + 11a^2b^2 +
    2ab^2 - b^2,\\
    g(a,b^2)&:=&a^8 - 4a^7 + 2a^6 + 8a^5 - 12a^4 + 20a^3b^2 + 6a^3 -
    30a^2b^2 - a^2 \\
    &&+ 14ab^2 - 2b^2,\\
    h(a,b^2)&:=&-3a^9 + \frac{27}{2}a^8 - 22a^7 + 14a^6 + a^5 - \frac{39}{2}a^4b^2 - 6a^4\\
     &&+ 39a^3b^2  +
    3a^3 - \frac{61}{2}a^2b^2 - \frac{1}{2}a^2 + 11ab^2 - \frac{3}{2}b^2.
 \end{eqnarray*}
Using the equations (\ref{re1}) and (\ref{re2}), a direct calculation shows relations (see \cite[7]{Github})
    \begin{eqnarray}
        f^2-fg+g^2 &=& h^2 \label{eq1}\\
        f+g-2h &=& 3ag \label{eq2}\\
        (g-f)(f-h)^2 &=& tg^3.\label{eq3}
    \end{eqnarray}
    If $g(a,b^2)$ is zero, it is straightforward from these equations that $f(a,b^2)=h(a,b^2)=0$. Then it is a calculation to check that  
    \begin{align*}
      & (-20a^3+30a^2- 14a+2)f(a,b^2)+( 21a^4- 32a^3+ 11a^2 +2a-1)g(a,b^2)\\
    &\qquad=81a^2(a-1)^2\left(a^2 - a + \frac{1}{3} \right)^4=0.
    \end{align*}
    From equation (\ref{re2}) we have $a \ne 0,1$ and therefore we have $a^2 - a + 1/3=0$, which contradicts the assumption that $a$ is $\fp$-rational. Thus $g(a,b^2)\ne 0$.
    Now we set $\lm:=f(a,b^2)/g(a,b^2)$. From equation (\ref{eq1}),
    we take $h(a,b)/g(a,b)$ to be a square root of ${\lm^2-\lm+1}$.  Then from equation (\ref{eq3}), we find that
    \[
    (1-\lm)(\lm-\sqrt{\lm^2-\lm+1})^2=t.
    \]
    This completes the proof.
\end{proof}
The choice of $\lm$ for a fixed $t$ in the proof is not unique. The next lemma tells us how many $\lm \in \fp$ satisfies $\Lm^{\varepsilon}(\lm)=t$ with $\varepsilon\in\{\pm1\}$.
\begin{lem}
\label{howmanylambda}
Let $p$ be a prime such that $p \equiv 11 \bmod 12$ and
let $E_t:y^2=x(x-1)(x-t)$ with $t \in \fp$ an elliptic curve. Then there is a one-to-one correspondence between the two sets
\begin{align*}
\left\{
 \begin{array}{cc}
       \lm \in \fp
 \end{array}
\middle|  
 \begin{array}{ccc}
       {\lm}^2-\lm+1 \in \left({\fp}^{*}\right)^2 ,\\
       \textrm{there exists an }\varepsilon\in \{\pm1\} \textrm{ such that }\\
      (1-\lm)(\lm+\varepsilon\sqrt{\lm^2-\lm+1})^2=t
 \end{array}
\right\}
\end{align*}
and 
\begin{align*}
    \{a \in \fp \mid  a \textrm{ is the } x\textrm{-coordinate of a point of } E_{t}[3]\}.
\end{align*}

\end{lem}
\begin{proof}
    We shall show the map sending $ \lm$ in the former set to
    \[\Phi(\lm)=\frac{1+\lm+2\varepsilon\sqrt{\lm^2-\lm+1}}{3}\in \fp\] is the bijection. We note that $\Lm^{-}(\lm) \ne \Lm^{+}(\lm)$ because $\lm^2-\lm+1\ne 0$ with $\lm\in \fp$. Thus $\varepsilon$ is uniquely determined respect to each $\lm$ in the former set. In addition, since $\lm^2-\lm+1\in\left(\fp^*\right)^2$, the image $\Phi(\lm)$ lies in $\fp$.
    From Lemma \ref{x-three}, the image $\Phi(\lm)=(1+\lm+2\varepsilon\sqrt{\lm^2-\lm+1})/3$ is the $x$-coordinate of a $3$-torsion point of $E_{\Lm^{\varepsilon}(\lm)}$, therefore the map is well-defined. Furthermore, let $\Psi$ be a map sending $a$ in the latter set to $\Psi(a)=f(a,b^2)/g(a,b^2)$, where $b^2:=a(a-1)(a-t)$ and $f(a,b^2),g(a,b^2)$ are as in Theorem \ref{important2}. Taking $h(a,b^2)/g(a,b^2)$ to be a square root of $\Psi(a)^2-\Psi(a)+1$ and $\varepsilon=-1$, by Theorem \ref{important2} we see that this map is well-defined. 
    Now we show that the two maps are inverse to each other. Firstly, from the equation (\ref{eq2}) in the proof of Theorem \ref{important2}, we have 
    \[
   \frac{1}{3} \left(\frac{f(a,b^2)}{g(a,b^2)}+1-2\frac{h(a,b^2)}{g(a,b^2)} \right)=a.
    \]
    This implies that $\Phi\circ \Psi=\text{id}$. Secondly, we observe $\Psi\circ\Phi$. Choose $\lm$ in the former set. A direct calculation (see \cite[8]{Github}) shows that $f(a_\lm,b^2_{\lm})/g(a_\lm,b^2_{\lm})=\lm$  where we write 
    \[
    a_\lm=\frac{\lm+1+2\varepsilon\sqrt{\lm^2-\lm+1}}{3}, \ \ {b_\lm}^2=A^{\varepsilon}(\lm)B^{\varepsilon}(\lm)^2
    \]
    with $A^{\varepsilon}(\lm),B^{\varepsilon}(\lm)$ in Lemma \ref{change_form3}. Thus we find that $\Psi\circ \Phi=\text{id}$. Now we obtain the desired result. 
\end{proof}
If the elliptic curve $E_t$ is supersingular, the cardinality of the latter set in Lemma \ref{howmanylambda} is determined.
\begin{lem}
\label{exactly2}
    Let $p$ be a prime such that $p \equiv 11 \bmod 12$ and
let $E_t:y^2=x(x-1)(x-t)$ with $t \in \fp$ a supersingular elliptic curve. Then
\[
\# \{a \in \fp \mid  a \textrm{ is the } x\textrm{-coordinate of point of } E_{t}[3]\}=2.
\]
\end{lem}
\begin{proof}
    Let $F$ be a $p$ th-power Frobenius map. Since $E_t$ is defined over $\fp$, a $3$-torsion point $P=(a,b)$ of $E_t(\overline{\fp})$ satisfies $a\in \fp$ if and only if $(1+F)(P)=0$ or $(1-F)(P)=0$. In addition, the assumption that $E_t$ is a supersingular elliptic curve defined over $\fp$ implies that  $\#\textrm{Ker}(1-F)=\textrm{deg}(1-F)=1+p$, and that $\#\textrm{Ker}(1+F)=\textrm{deg}(1+F)=1+p$. If we write $p=12m+11$ with some integer $m\ge 0$, we see that $p+1=2^2\cdot3(m+1)$. Thus each of the subgroups $\textrm{Ker}(1-F)$ and  $\textrm{Ker}(1+F)$ contains a subgroup of order $3$. We choose generators $P_1$ and $P_2$ for them, respectively. These points $P_1$ and $P_2$ is not contained in both of $\textrm{Ker}(1-F)$ and $\textrm{Ker}(1+F) $, so $P_1 \ne P_2$, and in particular, they generate $E[3]$. This means that the points of order $3$ of $\textrm{Ker}(1-F) $ are exactly $\{O,P_1,-P_1\}$ and that those of $\textrm{Ker}(1+F) $ are exactly $\{O,P_2,-P_2\}$. Since $P_1$ and $-P_1$ (resp. $P_2$ and $-P_2$) have the same $x$-coordinate, we get the desired result.  
\end{proof}

    \begin{lem}
\label{num_iso}
    Let $K$ be a field with char $K\ne2,3$, and let $\lm,\mu\in K\smallsetminus\{0,1\}$ with $\lm^2-\lm+1\ne0$ and $\mu^2-\mu+1\ne 0$.  Put $w(\lm):=16(\lm^2-\lm+1)^3/(\lm^2(\lm-1)^2)$, so that $j(E_\lm)=16\,w(\lm)$ for $E_\lm:y^2=x(x-1)(x-\lm)$, and
    \[
    [\mu]:=\left\{\mu, \frac{1}{\mu}, 1-\mu,\frac{1}{1-\mu}, \frac{\mu}{\mu-1},\frac{\mu-1}{\mu} \right\}.
    \]
    \begin{enumerate}
    \item[\rm{(a)}] With $w=w(\lm)$, the sum and the product of $j(E_{\Lm^{-}(\lm)})$ and $j(E_{\Lm^{+}(\lm)})$ are $\sigma_1(w)$ and $\sigma_2(w)$, respectively, where $\sigma_1(w):=w^3-126w^2+2944w$ and $\sigma_2(w):=w^4+576w^3+110592w^2+7077888w$.
    \item[\rm{(b)}] The equality
    \begin{equation}\label{eq:Lemma 6.6}
    \left\{j\left(E_{\Lambda^{-}(\lm)}\right),j\left(E_{\Lambda^{+}(\lm)}\right) \right\}=\left\{j\left(E_{\Lambda^{-}(\mu)} \right),j\left(E_{\Lambda^{+}(\mu)} \right)\right \}
    \end{equation}
    holds if and only if either $\lm \in [\mu]$, or $w(\lm)\ne w(\mu)$ and $\bigl(w(\lm)+w(\mu),\,w(\lm)w(\mu)\bigr)$ is one of $(576,2^{18})$, $(-32,8000)$ and $(-292,125000)$. In the latter case, both sides of \eqref{eq:Lemma 6.6} are the set of the roots of $P_{35}(X)$, $P_{20}(X)$ and $P_{32}(X)$, respectively.
    \item[\rm{(c)}] Let $K=\fp$ with $p\ge 5$, and assume that $E_{\Lm^{\pm}(\lm)}$ and $E_{\Lm^{\pm}(\mu)}$ are supersingular. Then \eqref{eq:Lemma 6.6} holds if and only if $\lm\in[\mu]$.
    \end{enumerate}
\end{lem}
    
    \begin{proof}
Recall that $w(\lm)=w(\mu)$ if and only if $\lm\in[\mu]$ (cf.\ \cite[III, Proposition 1.7]{AEC}). The assertion (a) follows from a direct calculation (see \cite[9-1]{Github}).

     (b) First assume that $\lm\in[\mu]$.  Then 
     the equality \eqref{eq:Lemma 6.6} follows from the equations
  \begin{align*}
       &1/\Lambda^{\varepsilon}(\lm)=\Lm^{-\varepsilon}\left( {\lm}/{(\lm-1)} \right), &
        &1-\Lm^{\varepsilon}(\lm)=\Lm^{-\varepsilon}(1-\lm), \\
        &{1}/({1-\Lm^{\varepsilon}(\lm)})=\Lm^{-\varepsilon}\left(({\lm-1})/{\lm}\right), &
       &({\Lm^{\varepsilon}(\lm)-1})/{\Lm^{\varepsilon}(\lm)}=\Lm^{-\varepsilon}\left({1}/({1-\lm} )\right), \\
        &{\Lm^{\varepsilon}(\lm)}/({\Lm^{\varepsilon}(\lm)-1}) = \Lm^{\varepsilon}\left({1}/{\lm}\right) & &
    \end{align*}
     for $\varepsilon=\pm1$, which is obtained by a straightforward computation.
{Next assume $\lm\notin[\mu]$, i.e., $w_1:=w(\lm)\ne w_2:=w(\mu)$, and put $e_1:=w_1+w_2$ and $e_2:=w_1w_2$. By (a), \eqref{eq:Lemma 6.6} means $\sigma_i(w_1)=\sigma_i(w_2)$ for $i=1,2$; dividing $\sigma_i(w_1)-\sigma_i(w_2)$ by $w_1-w_2$, we see that this is equivalent to
\[
e_2=e_1^2-126e_1+2944,\qquad e_1^3-2e_1e_2+576(e_1^2-e_2)+110592e_1+7077888=0,
\]
and eliminating $e_2$ we obtain $(e_1-576)(e_1+32)(e_1+292)=0$. The last assertion of (b) follows from $X^2-\sigma_1(w)X+\sigma_2(w)\equiv P_D(X) \bmod (w^2-e_1w+e_2)$ for $(e_1,e_2,D)=(576,2^{18},35),(-32,8000,20),(-292,125000,32)$ (see \cite[9-2]{Github}). 

(c) By (b) and Corollary \ref{peq7mod12}, we may assume $p\equiv 1 \bmod 4$ or $p\equiv 11 \bmod 12$, and have to exclude the latter case in (b), say with $D\in\{20,32,35\}$. For $p\in\{11,13,17,23,29\}$ this is checked by a direct computation (see \cite[9-3]{Github}); otherwise the roots $x_1,x_2$ of $P_D(X) \bmod p$ are distinct and not in $\{0,1728\}$, and $w_1,w_2\ne 108$. If $p\equiv 11 \bmod 12$, then $\Lm^{\pm}(\lm)\in\fp$ by Proposition \ref{pmod43fpcase}, so $w(\lm),x_1\in\fp$; this is impossible for $D=20,32$, as $w(\lm)=-16\pm 88\sqrt{-1}$, $-146\pm322\sqrt{-1}$, and for $D=35$ it gives $\left(\frac{-7}{p}\right)=\left(\frac{5}{p}\right)=1$ (note that the discriminant of $P_{35}(X)$ is $2^{32}5^37^223^2$; see \cite[20]{Github}), so that $p$ splits in $\bQ(\sqrt{-35})$ and $x_1$ is ordinary. Let $p\equiv 1\bmod 4$. Replacing $\mu$ by $1-\mu$ if necessary (note $1-\Lm^{\varepsilon}(\mu)=\Lm^{-\varepsilon}(1-\mu)$), we may identify $E:=E_{\Lm^{-}(\lm)}$ with $E_{\Lm^{-}(\mu)}$. As in the proof of Proposition \ref{end1}, $\alpha_\lm:=F\circ\psi^{-}_{\lm}$ and $\alpha_\mu:=F\circ\psi^{-}_{\mu}$ are endomorphisms of $E$ acting trivially on $E[2]$, where $\psi^-_\lm$, $\psi^-_\mu$ are $\psi^-$ of Remark \ref{isogenybetE} for $\lm$, $\mu$. Lifting $\lm$, $\mu$ to characteristic $0$ (to roots of $w(X)=\tilde w_1,\tilde w_2$ for the roots $\tilde w_1\ne\tilde w_2$ of $w^2-e_1w+e_2$), we see that $\psi^-_\lm$, $\psi^-_\mu$ are the reductions of $3$-isogenies between curves with complex multiplication by $O_D$ with distinct kernels, since $\Lm^{-}(\lm)$ and the kernel of $\psi^-_\lm$ determine $[\lm]$ (Lemma \ref{x-three} and the proof of Lemma \ref{howmanylambda}). Write $3O_D=\mathfrak a\overline{\mathfrak a}$; then $\mathfrak a^2=\beta O_D$ with $\beta=2\pm\sqrt{-5}$, $1\pm\sqrt{-8}$, $(1\pm\sqrt{-35})/2$ for $D=20,32,35$, and only $\tilde E[\mathfrak a]$, $\tilde E[\overline{\mathfrak a}]$ among the subgroups of order $3$ of $\tilde E$ give quotients with complex multiplication by $O_D$. Hence $\psi^-_\mu\circ[3]=\pm\psi^-_\lm\circ\beta$ for a suitable $\beta$, so $\alpha_\mu\circ[3]=\pm\alpha_\lm\circ\beta$, and $\pm\beta$ acts trivially on $E[2]$, i.e., $(\pm\beta-1)/2\in\End(E)\cap\bQ(\sqrt{-D})$. This intersection is $O_D$ ($O_D$ is maximal for $D=20,35$, and for $D=32$ a larger order would give $x_1=8000$, i.e., $p=13$), but $(\pm\beta-1)/2\notin O_D$, a contradiction.}
    \end{proof}

\section{Proof of Theorem \ref{thm:A} }
\label{pfofa}
In this section, we prove Theorem \ref{thm:A}
\subsection{The case of $p \equiv 1 \bmod 4 $ }

\begin{proof}[Proof of Theorem \ref{thm:A}]
From Proposition \ref{arigataya}, it suffices to count the number of $\lm \in \fp$ such that the two elliptic curves $E_{\Lm^{-}(\lm)}$ and $E_{\Lm^{+}(\lm)}$ are supersingular. From Theorem \ref{important}, those $\lm\in \fp$ are exactly $\fp$-elements $\lm$ such that the $j$-invariants $j(E_{\Lm^{-}(\lm)})$ and $j(E_{\Lm^{+}(\lm)})$ are the roots of $P_{3p}(X) \bmod p$. Since Theorem \ref{factorization_P3p} tells us that an $\fp$-root of $P_{3p}(X) \bmod p$ is $8000 \bmod p$ or $54000 \bmod p$, let $a_1:=8000, a_2:=54000$ and $b_{3,1},b_{3,2},\ldots, b_{n,1}, b_{n,2}\in \fpp$ be distinct roots of $P_{3p}(X) \bmod p$, where $b_{i,1}$ and $b_{i,2}$ are $\fp$-conjugate for $i=3,\ldots,n$. If $P_{3p}(X) \bmod p$ does not have the root $8000 \bmod p$ or $54000 \bmod p$, let $a_1:=0$ or $a_2:=0$ for convenience.  From Lemma \ref{num_iso} (c), the number of $\lm \in \fp$ satisfying $j(E_{\Lm^{-}(\lm)})=j(E_{\Lm^{+}(\lm)})=a_1$ is precisely $6$, the number satisfying $j(E_{\Lm^{-}(\lm)})=j(E_{\Lm^{+}(\lm)})=a_2$ is $3$, and the number satisfying $\{j(E_{\Lm^{-}(\lm)}),j(E_{\Lm^{+}(\lm)})\}=\{b_{i,1},b_{i,2}\}$ is $6$ for $i=3, \ldots, n$. Thus, we obtain
\[
\psi_p=6\epsilon_1+3\epsilon_2+6\epsilon_3+\cdots+6\epsilon_n,
\]
where $\epsilon_i:=1$. If $a_1=0$ or $a_2=0$, $\epsilon_1:=0$ or $\epsilon_2:=0$. 
In addition, Let $m_1,m_2,\ldots, m_n$ be multiplicities of each root $a_1,a_2,b_{3,1},b_{3,2},\ldots ,b_{n,1}, b_{n,2}$. Then we have
\[
h(-3p)=m_1\epsilon_1+m_2\epsilon_2+2m_3\epsilon_3+\cdots+2m_n\epsilon_n.
\]
By Theorem \ref{factorization_P3p}, we have $m_1=4$ and $m_2,\ldots ,m_n=2$ for $p\ne 5$. Hence we obtain $\psi_p=3h(-3p)/2$ for $p \ne 5$. If $p=5$, it easy to see that $\psi_p=\#\{-1,2,1/2\}$ and $h(-3\cdot 5)=2$. Thus we obtain $\psi_p=3h(-3p)/2$.
\end{proof}

\subsection{The case of $p \equiv 3 \bmod 4 $ }
 Theorem \ref{thm:A} in the case of $p \equiv7 \bmod 12$ follows from Corollary \ref{peq7mod12}, thus we may assume that $p \equiv11 \bmod 12$.

Let $x_0$ be the root of $P_p(X) \bmod p$. We recall that there is a $\lm\in \fp$ and $\varepsilon\in\{\pm1\}$ such that the $j$-invariant of $E_{\Lm^{\varepsilon}(\lm)}$  is equal to $x_0$ by Theorem \ref{important2}. In order to determine the value of $\psi_p$ for $p \equiv 11 \bmod 12$, we count the number of $\lm \in \fp$ such that $j(E_{\Lm^{\varepsilon}(\lm)})=x_0$ with some $\varepsilon\in\{\pm1\}$ for each root $x_0$ of $P_{p}(X) \bmod p$. Now we define a graph $G_p:=(V,E)$ together with a weight function $w:E\to\mathbb{R}$
 for each prime $p$ as follows: The vertex set $V$ consists of the roots of $P_p(X) \bmod p$. Two vertices $x_1$ and $x_2$ are joined by an edge if and only if there exists an element $\lm \in \fp$ satisfying 
\begin{eqnarray}
\label{lambdacondition}
    \{x_1,x_2\}=\{j(E_{\Lm^{-}(\lm)}),j(E_{\Lm^{+}(\lm)})\}.
\end{eqnarray}
In addition, the graphs are allowed to have self loops. The weight of the edge joining vertices $x_1$ and $x_2$ is defined to be the number of $\fp$-elements $\lambda$ satisfying the condition (\ref{lambdacondition}).


We show some properties of the graphs. 
\begin{lem}\label{propertiesofgraph}
Let $p>11$ be a prime with $p\equiv 11 \bmod 12$. Then the graph $G_p$ has the following properties. 
    \begin{enumerate}
    
    \item [\rm{(a)}] All vertices except those corresponding to $54000 \bmod p $ and $1728 \bmod p $ have degree $2$. 
        \item[\rm{(b)}] A vertex corresponding to $54000 \bmod p $ is the only vertex having a self loop. In addition, it is connected to another vertex by a single edge.
        \item [\rm{(c)}] A vertex corresponding to $1728 \bmod p $ is the only vertex having degree $1$. 
           \item [\rm{(d)}] 
           
           Assume that there is an edge jointing vertices $x_1$ and $x_2$ and that a $\mu \in\fp$ satisfies $\{x_1,x_2\}=\{j(E_{\Lm^{-}(\mu)}),j(E_{\Lm^{+}(\mu)})\}$. Then the weight of the edge is $\#[\mu]$ where $[\mu]$ is as in Lemma \ref{num_iso}. More precisely, the weight of the self loop of the vertex $54000 \bmod p$ is $3$ and weights of all the other edges are $6$.
        
    \end{enumerate}
\end{lem}
\begin{proof}
(a) We start with a vertex $x_0$ that is not $54000 \bmod p $ or $1728 \bmod p $.  Since $x_0$ is the root of $P_p(X) \bmod p$, we may assume that an elliptic curve with $j$-invariant $x_0$ is defined by an equation $y^2=x(x-1)(x-t)$ with some $t\in \fp$.  By Lemma \ref{howmanylambda} and Lemma \ref{exactly2}, there are exactly two distinct $\fp$-elements $\lm_1,\lm_2$ and $\varepsilon_1,\varepsilon_2 \in \{\pm1\}$ such that $\Lm^{\varepsilon_1}(\lm_1)$ and $\Lm^{\varepsilon_2}(\lm_2)$ are equal to $t$.  Furthermore, from the relations at the end of the proof of Lemma \ref{num_iso}, the set of $\fp$-elements $\lm$ that satisfy $j(E_{\Lm^{\varepsilon}(\lm)})=x_0$ with some $\varepsilon\in\{\pm1\}$ is $[\lm_1]\cup[\lm_2]$.  In addition, any element in $[\lm_1]$ (resp. $[\lm_2]$) gives the same edge from the vertex $x_0$. Our aim is to  show that the edges given by $[\lm_1]$ and $[\lm_2]$ are distinct, which occurs if and only if $\lm_2 \notin [\lm_1]$. Assume that $\lm_2\in [\lm_1]$. From Lemma \ref{num_iso}, the assumption that $\lm_2\in [\lm_1]$ and $\Lm^{\varepsilon_1}(\lm_1)=\Lm^{\varepsilon_2}(\lm_2)$ implies that $\Lm^{\varepsilon_1}(\lm_1)\in[1/2]$ or that $j\left(E_{\Lm^{-}(\lm_1)}\right)=j\left(E_{\Lm^{+}(\lm_1)}\right)$. The former case occurs only if $x_0\equiv1728\bmod p$. The latter case occurs only if $x_0\equiv54000 \bmod p$ from Proposition \ref{whenfp}. Thus, for vertices that do not correspond to $1728 \bmod p $ or $54000 \bmod p$, each of $[\lm_1]$ and $[\lm_2]$ gives distinct edge. Hence, the vertex $x_0$ has degree $2$.

(b) We note that the elliptic curve with $j$-invariant $54000\bmod p$ is supersingular when $p \equiv 11 \bmod 12$.  We know that $j(E_{\Lm^{-}(\lm)})\equiv j(E_{\Lm^{+}(\lm)})\equiv54000$ for $\lm\in[2]$ from Proposition \ref{whenfp}, so the vertex corresponding to $54000 \bmod p$ has a self loop. Proposition \ref{whenfp} also tells us that no other vertex has a self loop. In addition, from Lemma \ref{exactly2}, we have the other $\fp$-element $\mu\ne2$ such that $\Lm^{-}(\mu)=\Lm^{-}(2)$ or $\Lm^{+}(\mu)=\Lm^{-}(2)$. In order to show that $[2]$ and $[\mu]$ give distinct edges, it suffices to show that $\mu\notin [2]$. A direct calculation shows that $\Lm^{\pm}(1/2)=(2\pm\sqrt{3})/4$ and $\Lm^{\pm}(-1)=4(2\mp\sqrt{3})$ and then we see that they do not belong to $\{\Lm^{-}(2),\Lm^{+}(2)\}=\{-7\pm4\sqrt{3}\}$ for $p\ne 11$. Thus $\mu \notin[2]$.

(c)
We note that the elliptic curve $E$ with $j$-invariant $1728 \bmod p$ is supersingular when $p\equiv 11 \bmod 12$ and it can be given by the Legendre equation $y^2=x(x-1)(x+1)$. From Lemma \ref{exactly2}, we have exactly two $\fp$-elements $\lm_1,\lm_2$ such that $\Lm^{-}(\lm_i)$ or $\Lm^{+}(\lm_i)$ is equal to $1$ for $i=1,2$. The vertex $1728 \bmod p$ has only one edge if and only if $\lm_1$ and $\lm_2$ give the same edge, which occurs precisely when $\lm_1\in[\lm_2]$. We determine $\lm_1$ and $\lm_2$ explicitly from the $3$-torsion point of $E$ using Lemma \ref{exactly2}. The $x$-coordinate of $3$-torsion points are the roots of the following division polynomial of level $3$ of $E$:
\[
3x^4-6x^2-1=3\left(x^2-\frac{3+2\sqrt{3}}{3}\right)\left(x^2-\frac{3-2\sqrt{3}}{3}\right)=0.
\]
We note that $\pm\sqrt{3} \in \fp$ when $p \equiv11 \bmod 12$ and we see that
\[
\left( \frac{(3+2\sqrt{3})}{p} \right)\left( \frac{(3-2\sqrt{3})}{p} \right)=\left( \frac{-3}{p} \right)=-1.
\]
This means that either
$(3+2\sqrt{3})/3$ or $(3-2\sqrt{3})/3$ is a quadratic residue mod $p$ and then we find that the division polynomial has exactly two $\fp$-roots. Let $\pm a$ be the two $\fp$-roots of the division polynomial. Using Lemma \ref{exactly2}, two $\fp$-elements $\lm_1$ and $\lm_2$ are written as follows:
\begin{align*}
    \lm_1:=\frac{f(a,a(a-1)(a+1))}{g(a,a(a-1)(a+1))}, \quad \lm_2=\frac{f(-a,-a(a-1)(a+1))}{g(-a,-a(a-1)(a+1))}.
\end{align*}
We set
\begin{align*}
   && f_1(a):=f(a,a(a-1)(a+1)), \quad f_2(a):=f(-a,-a(a-1)(a+1)),\\
   && g_1(a):=g(a,a(a-1)(a+1)), \quad g_2(a):=g(-a,-a(a-1)(a+1)) .
\end{align*}
Then a direct calculation shows that (see \cite[17]{Github})
\[
f_1g_2-f_1f_2+g_1f_2=0,
\]
which shows that $\lm_1=\lm_2/(\lm_2-1)\in[\lm_2]$. This completes the proof.

(d) The first part of the lemma directly follows from Lemma \ref{num_iso} (c). In addition, since $\mu^2-\mu+1\ne 0$, $\#[\mu]=6$ for $\mu \notin [2]$ and $\#[\mu]=3$ for $\mu \in [2]$. From (b) we see that $\mu\in [2]$ corresponds to the self loop of the vertex $54000 \bmod p$. 
\end{proof} 
Now we are ready to prove Theorem \ref{thm:A} in the case of $p \equiv11 \bmod12$.
\begin{proof}[Proof of Theorem \ref{thm:A} in the case of $p \equiv11 \bmod12$] 
Firstly, we observe the case $p>11$.
    In order to determine the value of $\psi_p$, by Proposition \ref{arigataya}, we count the number of $\lm \in \fp$ such that the two elliptic curves $E_{\Lm^{-}(\lm)}$ and $E_{\Lm^{+}(\lm)}$ are supersingular. By Theorem \ref{important2}, those $\lm\in \fp$ are exactly the $\lm\in \fp$ such that the $j$-invariants $j(E_{\Lm^{-}(\lm)})$ and  $j(E_{\Lm^{+}(\lm)})$ are the roots of $P_{p}(X) \bmod p$. Furthermore, it is exactly equal to the summation of the weights of all edges of the graph $G_p$. Since $1728 \not \equiv54000 \bmod p$ for $p>11$ and from lemma \ref{propertiesofgraph}, we obtain the following equation:
    \[
    \psi_p=\sum_{e\in E}w(e)=6(n-1)+3=6n-3
    \]
    where $n$ is the number of distinct roots of $P_p(X) \bmod p$.
On the other hands, from Proposition \ref{fac_P_p}, the class number $h(-p)$ is written as follows:
\[
h(-p)=2(n-1)+1=2n-1.
\]
Hence we obtain $\psi_p=3h(-p)$ for $p>11$.
Finally, for $p=11$ it is easy to see that $\psi_p=3h(-p)=3$.  
\end{proof}

\section{Proof of Theorem \ref{thm:B}}
\label{pfofb}
We prove Theorem \ref{thm:B} in \textcolor{black}{the introduction}. We use the idea by Fouvry and Murty \cite[Theorem 4]{FM01}. 

Let $X$, $\epsilon$ and $N$ be as in Theorem B. Theorem \ref{thm:A} enables us to write the summations $\sum_{|\lm| \le N} {\phi}_{\lm}(X)$ with the class numbers as follows. 
\begin{eqnarray}
    \sum_{|\lm| \le N} {\phi}_{\lm}(X) &=&  \sum_{|\lm| \le N} \sum_{\substack{p < X \\  C_{\lm} \text{ is superspecial} }} 1 
= \sum_{p < X} \sum_{\substack{|\lm| \le N \\ C_{\lm} \text{ is superspecial} }} 1 \nonumber\\
&=& \sum_{p < X}   \left( \frac{2N}{p} +O(1)  \right) {\psi}_{p} \nonumber\\
 &=& \!\!\! \sum_{\substack{p<X \\ p \equiv 1 \bmod 4}} \frac{3N}{p} h(-3p) + \!\!\!  \sum_{\substack{p<X \\ p  \equiv 11 \bmod 12}} \frac{6N}{p} h(-p)  +  O \left( X^{\frac{3}{2}} \right) .\label{anaeq1}
\end{eqnarray}
Here and below $p$ runs over primes $p\ge5$; the primes $2,3$ contribute only $O(N)$. We note that the sum of $h(-p)$ and $h(-3p)$ over primes less than $X$ is $O(X^{3/2})$.
(cf. Mertens \cite{FM} and  Siegel \cite{SI}). 
Recall the Dirichlet class number formula
\[
h(-d)=\frac{\sqrt{d}}{ \pi}L(1,{\chi}_{-d}),
\]
where $\chi_{-d}$ is the character defined by the Kronecker symbol $(-d\ |\ \cdot)$ and $d>4$.
Then the sum \eqref{anaeq1} is  
\begin{align*}
\label{anaeq2}
  && \frac{2N}{\pi} \left(  \frac{3}{2}\sum_{\substack{p<X \\ p \equiv 1 \bmod 4}}\!\! \!\! \frac{\sqrt{3}}{\sqrt{p}}L(1,{\chi}_{-3p})+ \!\! \!\! \!\! \sum_{\substack{p<X \\ p \equiv 11\bmod 12 }} \!\! \!\!  \frac{3}{\sqrt{p}}L(1,{\chi}_{-p}) \right) + o \left( \frac{N\sqrt{X}}{\log X} \right).
  \stepcounter{equation}\tag{\theequation} 
\end{align*}
From Polya's inequality (cf. Apostol \cite[Theorem 8.21]{Apo}) and Abel's identity (cf. Apostol \cite[Theorem 4.2]{Apo}), for any parameter $U >1$, we have 
     \[
    \left| \sum_{U< n } \frac{{\chi}_{-p} (n)}{n} \right| < \frac{3\sqrt{p}\log p}{U}
    \]
    and the same inequality for $-3p$ with $\sqrt{p}\log p$ replaced by $\sqrt{3p}\log 3p$. 
We choose $U=X^{3/4}$. Then \eqref{anaeq2} is as follow\textcolor{black}{s}.
\begin{align*}
\label{anaeq3}
   \frac{6N}{\pi} \left\{ \frac{\sqrt{3}}{2} \!\!\!\!\sum_{\substack{p<X \\ p \equiv 1 \bmod 4}}\!\!\!\!\frac{1}{\sqrt{p}} \sum_{n \le U}\frac{{\chi}_{-3p}(n)}{n}+  \!\!\!\!\!\!\sum_{\substack{p<X \\ p \equiv 11 \bmod 12 }} \!\!\!\!\frac{1}{\sqrt{p}}\sum_{n \le U} \frac{{\chi}_{-p}(n)}{n} \right\}
    + o \left( \frac{N\sqrt{X}}{\log X} \right)
     \stepcounter{equation}\tag{\theequation} .
\end{align*}

    
We recall that ${\chi}_{-3p}$ and ${\chi}_{-p}$ are expressed by the Legendre symbols: 
$\displaystyle {\chi}_{-3p}= \left( \frac{n}{3} \right) \left( \frac{n}{p} \right)$
 if $p \equiv 1\bmod 4$
 and 
 $\displaystyle {\chi}_{-p}= \left( \frac{n}{p} \right)$
if $p \equiv 3 \bmod 4$.
Substituting these \textcolor{black}{in} \eqref{anaeq3} yields
\begin{align*}
\label{anaeq4}
   && \frac{3\sqrt{3}N}{\pi}S(X,U) +  \frac{6N}{\pi} T(X,U)
    + o \left( \frac{N\sqrt{X}}{\log X} \right),
     \stepcounter{equation}\tag{\theequation} 
\end{align*}
where
\begin{eqnarray*}
    S(X,U)&:=& \sum_{n \le U} \frac{1}{n} \sum_{\substack{p<X \\ p \equiv1 \bmod 4}}\frac{\left( \frac{n}{3} \right) \left( \frac{n}{p} \right)}{\sqrt{p}},\\
    T(X,U)&:=&\sum_{n \le U} \frac{1}{n} \sum_{\substack{p<X \\ p \equiv 11\bmod 12}}\frac{\left( \frac{n}{p} \right)}{\sqrt{p}}.
\end{eqnarray*}
We estimate these sums via the idea of Fouvry and Murty \cite[Theorem 4]{FM01}.

We first calculate $T(X,U)$. If we sum over only $n \le U$ such that $n$ is a perfect square or $n/3$ is a perfect square, then we have 
    \begin{eqnarray*}
    \sum_{\substack{n \le U \\ n=d^2 \text{ for some } d \in \bN }} \frac{1}{d^2} \sum_{\substack{p<X \\ p \equiv 11 \bmod 12 }} \frac{1}{\sqrt{p}}+\sum_{\substack{n \le U \\ n=3l^2 \text{ for some } l \in \bN }} \frac{1}{3l^2} \sum_{\substack{p<X \\ p \equiv 11 \bmod 12}} \frac{1}{\sqrt{p}}.
    \end{eqnarray*}
we note that $3 \in {{\fp}^{*}}^2$ if $p \equiv 11 \bmod 12$. 
We know that the inner sum of each term is equal to
\[
\frac{2}{\varphi(12)} \frac{\sqrt{X}}{\log X} +o \left( \frac{\sqrt{X}}{\log X} \right),
\]
where $\varphi$ is Euler's totient function. 
Thus, we get
\[
 \sum_{\substack{n \le U \\ n=d^2 \text{ for some } d \in \bN }} \frac{1}{d^2} \sum_{\substack{p<X \\ p \equiv 11 \bmod 12 }} \frac{1}{\sqrt{p}}=\frac{{\pi}^2}{12} \frac{\sqrt{X}}{\log X} +o \left( \frac{\sqrt{X}}{\log X} \right)
\]
and
\[
\sum_{\substack{n \le U \\ n=3l^2 \text{ for some } l \in \bN }} \frac{1}{3l^2} \sum_{\substack{p<X \\ p \equiv 11 \bmod 12 }} \frac{1}{\sqrt{p}}=\frac{{\pi}^2}{36} \frac{\sqrt{X}}{\log X} +o \left( \frac{\sqrt{X}}{\log X} \right).
\]
\textcolor{black}{Following a similar approach to} Fouvry and Murty \cite[p.\ 92]{FM01} and using Lemma \ref{jutila}, the sum of the remaining terms is ultimately negligible compared to $\sqrt{X}/ \log X$.
Hence we obtain
\begin{eqnarray}
\label{11mod12}
    T(X,U)=\frac{{\pi}^2}{9} \frac{\sqrt{X}}{\log X} +o \left( \frac{\sqrt{X}}{\log X} \right).
\end{eqnarray}

Second, we calculate $S(X,U)$. If $3$ divides $n$, then ${\chi}_{-3p}(n)$ is $0$ so we only treat $n$ such that $3 \nmid n $. If we sum over only such $n \le U$  that is also perfect square, then we have
  \begin{eqnarray*}
   && \sum_{\substack{n \le U \\ n=d^2 \text{ for some } d \in \bN, \ 3 \nmid d}} \frac{1}{d^2} \sum_{\substack{p<X \\ p \equiv 1  \bmod 4}} \frac{1}{\sqrt{p}}\\ &&=\left(\sum_{\substack{n \le U \\ n=d^2 \text{ for some } d \in \bN}} \frac{1}{d^2} -\sum_{\substack{n \le U \\ n=(3d)^2 \text{ for some } d \in \bN}} \frac{1}{(3d)^2}\right)\sum_{\substack{p<X \\ p \equiv 1  \bmod 4}} \frac{1}{\sqrt{p}}.
    \end{eqnarray*}
    We know that 
    \begin{eqnarray*}
     && \sum_{\substack{p<X \\ p\equiv 1 \bmod 4}} \frac{1}{\sqrt{p}}=  \frac{2}{\varphi(4)} \frac{\sqrt{X}}{\log X} +o \left( \frac{\sqrt{X}}{\log X} \right).
    \end{eqnarray*}
    Thus we obtain 
    \begin{eqnarray*}
         \sum_{\substack{n \le U \\ n=d^2,\ \exists d \in \bN,\ 3 \nmid d }} \frac{1}{d^2} \sum_{\substack{p<X \\ p \equiv 1 \bmod 4 }} \frac{1}{\sqrt{p}}=\frac{{4\pi}^2}{27}\frac{\sqrt{X}}{\log X} +o \left( \frac{\sqrt{X}}{\log X} \right).
    \end{eqnarray*}
\textcolor{black}{Following a similar approach to} Fouvry and Murty \cite[p.92]{FM01} and using \cite[Lemma 6]{FM01} for $n\equiv 1\bmod 4$, the sum of the remaining terms is ultimately negligible compared to $\sqrt{X}/ \log X$.  Thus we find that
\begin{eqnarray}
\label{1mod4}
    S(X,U)=\frac{4{\pi}^2}{27}\frac{\sqrt{X}}{\log X} +o \left( \frac{\sqrt{X}}{\log X} \right).
\end{eqnarray}
Substituting the equations \eqref{11mod12}, \eqref{1mod4} for \eqref{anaeq4} yields 
\[
 \sum_{|\lm| \le N} {\phi}_{\lm}(X) =\frac{(6+4\sqrt{3})\pi N}{9}\frac{\sqrt{X}}{\log X} +o \left( \frac{N \sqrt{X}}{\log X} \right).
\]
Hence we obtain
\[
 \frac{1}{N}\sum_{|\lm| \le N} {\phi}_{\lm}(X) =\frac{(6+4\sqrt{3})\pi }{9} \frac{\sqrt{X}}{\log X} +o \left( \frac{ \sqrt{X}}{\log X} \right).
\]
\section{Proof of Theorem \ref{thm:C}}
\label{pfofc}
In this section, we prove Theorem \ref{thm:C}.  From our previous work \cite[\S 6]{Ando}, we have
\begin{eqnarray*}
    \sum_{\text{ht}(\lm) \le N} {\phi}_{\lm}(X) &=& 
\sum_{p < X} \sum_{\substack{\text{ht}(\lm) \le N \\ C_{\lm} \text{ is superspecial} }} 1\\
&=& \sum_{p < X} \left(  \frac{12}{{\pi}^2} \frac{N^2}{p} + O(N\log N) +O\left( \frac{N^2}{p^2} \right) \right) {\psi}_p.
\end{eqnarray*}
Since $\sum_{p<X}\psi_p=O(X^{3/2})$, $\sum_{p}\psi_p/p^2<\infty$ and $N>X^{1+\epsilon}$, the error terms are $o(N^2\sqrt{X}/\log X)$. Thus, from the result of Theorem B we find that
\[
\sum_{\lm \in \bQ,\ \height(\lm) \le N } {\phi}_{\lm}(X) = \frac{4(3+2\sqrt{3})}{3\pi} \frac{N^2\sqrt{X}}{\log X} +o \left( \frac{N^2\sqrt{X}}{\log X}\right).
\]
Hence we obtain
 \[
    \frac{1}{N^2}\sum_{\lm \in \bQ,\ \height(\lm) \le N } {\phi}_{\lm}(X) = \frac{4(3+2\sqrt{3})}{3\pi} \frac{\sqrt{X}}{\log X} +o \left( \frac{\sqrt{X}}{\log X}\right).
 \]


\bigskip

\noindent{\bf Data availability statement:} 
Data sharing is not applicable to this article, as no datasets were generated or analyzed during the present work. The Magma codes used for the computations in this paper are available in the GitHub repository \cite{Github}.

\medskip

\noindent{\bf Conflict of interest:}
The authors declare no conflicts of interest associated with this manuscript.

\medskip

\noindent{\bf Funding:}
This research received no specific grant from any funding agency in the public, commercial, or not-for-profit sectors.

\end{document}